\documentclass{amsart}
\usepackage{amsmath}
  \usepackage{paralist}
  \usepackage{graphics} %% add this and next lines if pictures should be in esp format
  \usepackage{epsfig} %For pictures: screened artwork should be set up with an 85 or 100 line screen
\usepackage{graphicx}  \usepackage{epstopdf}%This is to transfer .eps figure to .pdf figure; please compile your paper using PDFLeTex or PDFTeXify.
 \usepackage[colorlinks=true]{hyperref}
 \usepackage{amsmath}

\usepackage{mathrsfs}
\usepackage{graphicx}
\usepackage{fancyhdr}
\usepackage{textcomp}
\usepackage{amsthm}
\usepackage{amscd}
\usepackage{latexsym}
\usepackage{mathrsfs}
\usepackage{amsfonts}
\usepackage{amssymb}
\usepackage{appendix}
\usepackage{tikz}
\usepackage{upgreek}
\usepackage{minitoc}
\hypersetup{urlcolor=blue, citecolor=red}

\newcommand{\opnm}{\operatorname}

\newcommand{\BigO}{\mathcal{O}}

\newcommand{\ii}{\textnormal{i}}

\newcommand{\ee}{\textnormal{e}}

\newcommand{\CH}{\mathscr{C}}
\newcommand{\SH}{\mathscr{S}}
\newtheorem{thm}{Theorem}[section]
\newtheorem{prop}[thm]{Proposition}
\newtheorem{cor}[thm]{Corollary}
\newtheorem{lem}[thm]{Lemma}

\newtheorem{remarque}[thm]{Remark}
\theoremstyle{defi}

\title[Study of the Kramers-Fokker-Planck-magnetic quadratic operator ] %Use the shortened version of the full title
      {Study of the Kramers-Fokker-Planck quadratic operator with a constant magnetic field}

\author[Zeinab Karaki]{}


 \email{zeinab.karaki@univ-nantes.fr}
 \date{1 March 2022}

\begin{document}
\maketitle

% Enter the first author's name and address:
\centerline{\scshape Zeinab Karaki}
\medskip
{\footnotesize
% please put the address of the first author
 \centerline{Universit\'e de Nantes}
   \centerline{Laboratoire de Mathematiques Jean Leray}
  \centerline{ 2, rue de la Houssini\`ere}
   \centerline{BP 92208 F-44322 Nantes Cedex 3, France}
%} % Do not forget to end the {\footnotesize by the sign }

\medskip

\bigskip

% The name of the associate editor will be entered by an editorial staff
% "Communicated by the associate editor name" is not needed for special issue.
 %\centerline{(Communicated by the associate editor name)}

%The abstract of your paper
 \begin{abstract}
We study the quadratic Kramers-Fokker-Planck operator with a constant magnetic field and with a quadratic potential. We describe the exact expression of the norm of the semi-group associated to the operator near the equilibrium. At this level, explicit and accurate estimates of this norm are shown in small and long times as well as uniform-in-time estimates when the magnetic parameter $b$ tends to infinity.
\end{abstract}
%\begin{keywords}
%return to equilibrium, hypoellipticity; semigroup; hypocoercivity; hypodissipativity; Fokker-Planck equation; magnetic field; Lorentz force;
%enlargement space.
%\end{keywords}

%\tableofcontents
\section{Introduction and main results}
\subsection{Presentation of the operator}
We consider $ P $ the quadratic Kramers-Fokker-Planck operator with a constant external magnetic field $ B_e \in \mathbb{R}^3 $ and a linear isotropic electric field $ E_{e} (x): = ax $ with $ a> 0 $
\begin{align*}
P=(-\nabla_v +v/2)\cdot (\nabla_v +v/2) +v\cdot \nabla_x +E_e\cdot \nabla_v +(v\wedge B_e)\cdot \nabla_v.
\end{align*}
where $ v \in \mathbb{R}^{3} $ represents the velocity, $ x \in \mathbb{R}^{d} $ represents the space variable and $ t \geq $ 0 is the time.
We note that in dimension $3$, by rotation we can always reduce to a vector field of type $ (0,0, b) $, which allows us to take $ B_e \in \{(0 , 0, b) \, / \, b \in \mathbb{R} \} $. In addition, in dimension three, we know that there is a direction that does not depend on the magnetic field (direction parallel or anti-parallel to the magnetic field from which the magnetic effect is zero). This justifies a restriction to a model in dimension $2$.

By changing variables (see \cite[Section 4.1]{zk2} for more details), the dimension-$2$ version of the operator $ P $ can be written in the following symmetric form:
\begin{align}\label{def P}
2P_{a,b}:= (-\nabla_v +v)\cdot (\nabla_v +v) + 2\,\sqrt{a} (v\cdot\nabla_x -x\cdot \nabla_v)+b (v_1\partial_{v_2}-v_2\partial_{v_1}) . 
\end{align}
Now, we decompose the operator $ P_ {a, b} $ according to the operators of creation and annihilation (cf.\ \cite[Theorem 1.4]{viola2013}). H. Risken in \cite{risken1998fokker} has proposed a decomposition into operators of creation and annihilation in the case of the Kramers-Fokker-Planck operator without a magnetic field, and it turns out that we are able even with a magnetic field to write $P_{a,b}$ in the following form:
\begin{align} \label{expression op-1}
  P_{a,b}=\frac{1}{2}\sum_{j,k=1}^{4}\, m_{jk}A_{k}^{*}A_{j},
  \end{align}
  where  $ (m_{j,k})_{j,k}$ are the coefficients of the matrix $ M_{a, b} $ defined by
   \begin{align}\label{matrice-1}
 M_{a,b}= \begin{pmatrix}
 1 & b & \sqrt{a} & 0 \\
 -b & 1 & 0 &\sqrt{a}\\
 -\sqrt{a} & 0 & 0 & 0\\
 0 & -\sqrt{a} & 0 & 0
 \end{pmatrix}
 \end{align}
 and the annihilation operators $ (A_j)_{j = 1}^{4} $ and the creation operators
$ (A_j^{*})_{j = 1}^{4} $ are
\begin{align*}
&A_j=\partial_{x_{j}} + x_{j} , && A_{j}^{*}=-\partial_{x_{j}}+x_{j}&&&\forall j=1,2, \\
& A_j=\partial_{v_{j-2}} + v_{j-2}&& A_{j}^{*}=-\partial_{v_{j-2}}+v_{j-2} &&& \forall j=3,4.
\end{align*}
The advantage of this decomposition is that we can apply the abstract theory made for this type of operators. We refer to the work of Aleman-Viola in \cite{aleman2014singular}, which consists, for example, in giving an exact description of the spectrum of the operator according to the spectrum of its associated matrix (see also \cite[Thm 1.4]{viola2013}, \cite[Section 1]{aleman2018weak} and \cite{sjostrand1974parametrices, hitrik2011resolvent}). For explicit calculations of the Kramers-Fokker-Planck quadratic operator spectrum without a magnetic field, we refer readers to the book of B. ~ Helffer and F. ~ Nier \cite[Section 5.5.1]{nier2005hypoelliptic}.

It is of interest to study the asymptotic behavior of the solution of the evolution equation associated with the non-self-adjoint operator $P_{a,b}$
 $$ \begin{cases}
 \partial_t u(t,x)+P_{a,b}\,u(t,x)=0\\
 u_{\mid_{t=0}}=u_0\in L^2 (\mathbb{R}^n)\,.
 \end{cases}$$
  More specifically, we are interested in the question of return to equilibrium.
  This usually amounts to studying the following operator:
  $$ \ee^{-tP_{a,b}}(1-\Pi_0) \,$$
  where $ \Pi_0 $ is the spectral projector associated with the eigenvalue zero (typically it represents the orthogonal projector on the vector subspace generated by the Maxwellian associated with the problem) and where $ \ee^{- t \,P_{a, b}} $ represents the semi-group associated with the operator $ P_ {a, b} $. Many works \cite{aleman2018weak, aleman2014singular} lead to the question of return to equilibrium starting from the study of the asymptotic norm of the matrix exponential $ \ee^{- tM_ {a, b} } $. This can be done by using the following equality:
  $$ \Vert \ee^{-tP_{a,b}}(1-\Pi_0)\Vert_{\mathcal{B}(L^2(\Bbb R^4))}=\Vert \ee^{-tM_{a,b}}\Vert\,, $$ 
  where $ \Vert \cdot \Vert_{\mathcal{B} (L^2 (\Bbb R^4))} $ denotes the norm on the set $ \mathcal{B} (L^2 (\Bbb R^4)) $ of bounded operators on $ L^2 (\Bbb R^4) $ and $ \Vert\cdot \Vert $ denotes the standard norm on matrices as linear operators, induced by the Euclidean norm on $ \mathbb{C}^4$. We note that this last equality is shown in \cite[Corollary 7]{aleman2014singular}.
  
  The calculation of the previous norm in terms of the parameters $a$ and $b$ represents the heart of our present work. In this paper, we will continue the study of the model case of the Fokker-Planck operator with an external magnetic field, begun in a wider than quadratic framework in \cite{ZK1} for purely kinetic study using the method of hypocoercivity, which was developed by F.~ H\'erau \cite{herau2004isotropic} and C.~Villani \cite{dric2009hypocoercivity}, and the method of factorization and enlargement of functional spaces, which is addressed by M.~Gualdani {\it et al.} \cite{gualdani2010factorization}. Then, we refer readers to the article \cite{zk2} for a demonstration of maximal estimates on this model, giving a characterization of the domain of its closed extension based on extremely sophisticated maximal hypoellipticity tools developed by B.~ Helffer and J.~ Nourrigat in the 1980s (see \cite{helffer1980hypoellipticite}).
  
  We will study the case of the quadratic operator of Kramers-Fokker-Planck with a constant external magnetic field. First, we explain the exact expression of the exponential norm of the matrix $ M_{a, b} $, using the Lagrange interpolation method and some properties of symmetries that we will notice in the algebraic structure associated to $ M_{a, b} $. Then, we deduce precise estimates of the norm of the solution of the equation of evolution associated with the operator $ P_{a, b} $ at several levels. We measure in particular the behavior in small and long time and uniformly in $ t $, and we study asymptotics when the magnetic field tends to infinity.
  \subsection{The main results}
The first result of this work is an exact description of the spectrum of the operator $ P_{a, b} $ and the behavior, when $ b $ goes to infinity of its spectral gap which equals $ \min \{\Re \lambda, \, \lambda \in \Sigma (P_{a, b}) \setminus \{0 \}) \} = \min \{\Re \lambda, \, \lambda \in \Sigma (M_{ a, b}) \} $ where $ \Sigma (P_{a, b}) $ denotes the spectrum of $ P_{a, b} $.
\begin{thm}\label{thm q1}
Let $ a> 0 $ and $b \in \Bbb{R}$. Then the generalized eigenvalues of the operator $ P_ {a, b} $ in \eqref{def P} are 
  $$\alpha_{k}=\sum_{j=1}^{4}\,\lambda_{j}k_{j}, \quad \forall k\in \mathbb{N}^{4}, $$
  with $ \{\lambda_{j} \}_{j = 1}^{4} $ the eigenvalues of $ M_{a, b} $, which are defined in \eqref{2} and \eqref{3} . In addition,
  \begin{enumerate}
\item[1)] When $b\to +\infty $, we have
\begin{align*}
&\lambda_{1}=\frac{a}{b^2}+\mathcal{O}(\langle a\rangle^2 b^{-4} )- \ii (-\frac{a}{b}+\mathcal{O}(\langle a\rangle^2 b^{-3} )) ,\\
&\lambda_{2}=1-\frac{a}{b^2}+\mathcal{O}(\langle a\rangle^2 b^{-4} )- \ii (b+\frac{a}{b}+\mathcal{O}(\langle a\rangle^2 b^{-3} )),\\
&\lambda_{3}=\frac{a}{b^2}+\mathcal{O}(\langle a\rangle^2 b^{-4} )+\ii (-\frac{a}{b}+\mathcal{O}(\langle a\rangle^2 b^{-3} )) ,\\
&\lambda_{4}=1-\frac{a}{b^2}+\mathcal{O}(\langle a\rangle^2 b^{-4} )+ \ii (b+\frac{a}{b}+\mathcal{O}(\langle a\rangle^2 b^{-3} )).
\end{align*}
where $\langle a\rangle=(1+a^2)^{1/2}$.
\item[2)] If $ b \neq 0 $, for all $ \lambda_{j} \in \Sigma (M_{a, b}) $ then $ \Im (\lambda_{j}) \neq 0 $ .
  \item[3)] The spectrum of $ M_{a, b} $ does not depend on the sign of $ b $.
\end{enumerate} 
\end{thm}
We are interested in the explicit computation of the norm of the matrix exponential $ \ee^{- t \, M_{a, b}} $ as a function of the following auxiliary quantities:  
  \begin{align}
\label{eq_def_A} A &= A_1 + \ii A_2 = 1 - b^2 - 4a - 2\ii b\\
\label{eq_def_c} c &= c_1 + \ii c_2 = \sqrt{A},
\end{align}
where $A_1,A_2,c_1,c_2\in \Bbb R$.  
\begin{thm}\label{thm_2,4}
For $a, b \in \Bbb{R}$ and $t > 0$, let $A = A_1 + \ii A_2$ and $c = c_1 + \ii c_2$ be as in \eqref{eq_def_A} and \eqref{eq_def_c}. When
\begin{equation}\label{eq_def_T}
		T = \frac{1}{2}\left((|A| - A_1 + 2)\cosh c_1 t + (|A| + A_1 - 2)\cos c_2 t\right)
\end{equation}
and
\begin{equation}\label{eq_def_S}
\begin{aligned}
	S &= 8a(1-\Re(\cosh ct)) + \frac{1}{4}(2-A_1)|A|(\cosh 2c_1 t - \cos 2c_2 t)
	\\ & \qquad + (2a + \frac{1}{4}|A|^2)(\cosh 2c_1 t + \cos 2c_2 t - 2),
\end{aligned}
\end{equation}
the norm of the exponential of $M_{a,b}$ defined in \eqref{matrice-1} can be computed as
\[
	\|\ee^{-tM_{a,b}}\|^2 = \frac{1}{|A|}\ee^{-t}\left(T + \sqrt{S}\right).
\]
\end{thm}
\begin{remarque}
While the expression for $\|\ee^{-tM_{a,b}}\|^2$ is fairly complicated, a closed form for the exponential of a $4\times 4$ matrix is rarely so simple. One sees immediately that $T$ and $S$ may be expanded in series of even powers of $t$ where the coefficient of $t^0$ is $|A|$ in the expansion of $T$ and the coefficient of $t^0$ is zero in the expansion of $S$.
\end{remarque}
We can deduce precise estimates. First, when $ t \to 0^+$, we have a complete asymptotic expansion for the exponential norm of $ M_{a, b} $.
\begin{prop}\label{prop_Taylor}
The norm $\|\ee^{-tM_{a,b}}\|$ admits a complete asymptotic expansion in powers of $t$ as $\langle A \rangle t^2 \to 0^+$ beginning with 
\[
	\|\ee^{-tM_{a,b}}\| = 1 - \frac{a}{12}t^3 + \left(\frac{1}{360}ab^2 + \frac{1}{240}a^2 + \frac{1}{120}a\right)t^5 + \frac{1}{288}a^2 t^6 + \BigO(\langle A\rangle^3 t^7).
\]
\end{prop}
Second, as $|b| \to \infty$ in such a way that $\langle a\rangle b^{-2} \to 0$, we have uniform estimates on the difference between $\ee^{(1-c_1)t}\|\ee^{-tM_{a,b}}\|^2$ and $1$ for all $t \geq 0$.
\begin{prop}\label{prop_large_b}
There exists some $C > 0$ such that, if $a > 0$, $b \neq 0$ and $\langle a \rangle b^{-2} < \frac{1}{C}$, then
\[
	\left|\ee^{(1-c_1)t}\|\ee^{-tM_{a,b}}\|^2-1\right| \leq C\langle a\rangle^2 b^{-4}.
\]
\end{prop}
  This measures in a sense how non-self-adjoint $M_{a,b}$ is: since $\frac{1}{2}(1-c_1)$ is the spectral abscissa $\min\{\Re \lambda \::\: \lambda \in \opnm{Spec} M_{a,b}\}$, if $M$ were self-adjoint one would have the equality $\|\ee^{-tM_{a,b}}\|^2 = \ee^{-(1-c_1)t}$ for all $t \geq 0$.
  
   Finally, we identify the exact value of $\lim\limits_{t \to \infty}\ee^{(1-c_1)t}\|\ee^{-tM_{a,b}}\|^2$ which is related to the norms of spectral projections of $M_{a,b}$.
\begin{prop}\label{prop_long_t}
For $a, b > 0$. There exists some $C > 0$ such that if
\[
	E(t) = \ee^{-2c_1t} + \langle a\rangle^2 b^{-4}\ee^{-c_1t} \leq \frac{1}{C},
\]
then
\[
	\left|\left( \frac{c_1^2+c_2^2}{c_2^2+1}\right)^{1/2}\ee^{\frac{t(1-c_1)}{2}}\|\ee^{-tM_{a,b}}\|-1\right| \leq CE(t).
\]
\end{prop}
\begin{remarque}
We note that according to the previous Proposition, we have
$$  \lim\limits_{t\to +\infty} \ee^{\frac{t(1-c_1)}{2}}\|\ee^{-tM_{a,b}}\|=\sqrt{\frac{c_2^2+1}{c_1^2+c_2^2}} =:\sqrt{R_1},$$
and this last constant coincides with the norm of any spectral projector $ \Pi_j $ associated with $ \lambda_j  \in \Sigma(M_{a,b})$
$$ \sqrt{R_1}=\Vert \Pi_1\Vert=\Vert\Pi_2\Vert=\Vert \Pi_3\Vert= \Vert \Pi_4\Vert.$$
These norms can be calculated as $$\Vert \Pi_j\Vert=\displaystyle \frac{\Vert v_j\Vert\Vert \omega_j\Vert}{\langle v_j, \omega_j\rangle}=\frac{a+\vert \lambda_j^2\vert}{\vert a-\lambda_j^2\vert}, \,\forall j=1,\dots ,4$$
where $M_{a,b}\,v_j=\lambda_j v_j$ and $M_{a,b}^*\,\omega_j=\overline{\lambda}_j\omega_j$. 
\end{remarque}
We conclude this part with a brief review of the literature related to the analysis of quadratic operators, more specifically of the Kramers-Fokker-Planck (KFP) type. In recent years, several works have been focused on this operator with diversified approaches. Numerous authors have been interested in proving maximal (or global) estimates to deduce the compactness of the resolvent of the (KFP) operator's in order to address the question of returning to equilibrium.
F. H\'erau and F. Nier in the article \cite{herau2004isotropic} put the links between the (KFP) operator with a confining potential and the associated Witten Laplacian operator. Then, in the book of B. ~ Helffer and F. ~ Nier \cite{nier2005hypoelliptic}, these works have been completed and explained in a general way, and we refer more specifically to the section $ 5.5 $ for a specific study of the operator of (KFP) with a quadratic potential where they calculated the exact spectrum of the quadratic operator, where the calculations are based on the book of H. ~ Risken \cite{risken1998fokker}.
More recently, M. ~ Ben Said {\it et al.} in \cite{said2018quaternionic} have studied global estimates for model operators of (KFP) with polynomial potentials of degree less than or equal to $ 2 $, then M. ~ Ben Said in \cite{said2018global} provided accurate global subelliptic estimates for (KFP) operators on a more general case involving a certain class of polynomials of degree greater than $ 2 $.

\textbf{Plan of the article:} This article is organized as follows. In section 2, we will prove Theorem \ref{thm q1}. Then, in section 3 we will determine the explicit form of the matrix exponential $ \ee^{- tM_{a, b}} $ by applying the Lagrange interpolation method and we will thus give the proof of Theorem \ref{thm_2,4}. Finally, section 4 is dedicated to the proof of Propositions \ref{prop_Taylor}, \ref{prop_large_b} and \ref{prop_long_t}.
\section{Spectrum of the operator $ P_ {a, b} $}
\subsection{Preliminary.}
To determine the spectrum of operators, acting on $ L^2 (\mathbb{R}^n) $ and decomposable into annihilation and creation operators $ A_j = \partial_{x_{j}} + x_j $ and $ A_{j}^{*} = - \partial_{x_{j}} + x_j $ as follows :
  $$P=\frac{1}{2}\sum_{j,k=1}^{n}\, m_{jk}A_{k}^{*} A_{j}\,,$$
    where $M = (m_{jk} )_{j,k=1}^{n} \in M_{n\times n} (\mathbb{C})$,   
    it is equivalent to determine the spectrum of the matrix $ M $ by applying the following result which gives the link between the spectrum of the operator and its associated matrix :
    \begin{cor}[Corollary 2.3 in \cite{aleman2014singular}]\label{cor q1}
    Let $ \{\lambda_{j} \}_{j = 1}^{n} $ be the eigenvalues of the matrix $ M $, repeated by their multiplicity. Let $ G $ such that $ GMG^{- 1} $ is in Jordan's normal form, thus having diagonal entries $ \lambda_{1}, \dots, \lambda_{n} $. So
    $$\{\mathcal{B}^{*}((Gz)^{\alpha})\}_{\alpha \in  \mathbb{N}^{n}} $$
    form a system of eigenfunctions of the operator $ P $ associated with generalized eigenvalues
    $$\lambda_{\alpha}=\sum_{j=1}^{n}\,\lambda_{j}\alpha_{j},\quad \forall \alpha\in \Bbb N^{n}\,, $$
    where $ \mathcal{B} $ denotes the Bargmann transform
which is defined as follows:
$$\mathcal{B}f(z)=\pi^{-3n/4}\, \int_{\mathbb{R}^{n}}\,f(x)\,\ee^{\sqrt{2}xz-x^{2}/2-z^{2}/2}\,dx \,.$$
 \end{cor}
 We note that A.~Aleman and J.~Viola in the article \cite{aleman2014singular} have shown the previous Corollary by following essentially the work of J.~Sj\"ostrand in \cite[Theorem 3.5]{sjostrand1974parametrices}. (see also \cite[Lemma 4.1]{hitrik2011resolvent} and \cite{hitrik2009spectra}).
 \subsection{Proof of Theorem \ref{thm q1}.}
Note that our operator is decomposable into annihilation and creation operators in \eqref{expression op-1}. According to Corollary \ref{cor q1}, one can determine the spectrum of $ P_{a, b} $ as functions of the elements of the spectrum of its associated matrix $ M_{a, b} $. First, we are interested in looking for the spectrum of the matrix $ M_{a, b} $. 
 \begin{lem}\label{lem q1}
 Let $ a> 0 $ and $ b \in \mathbb{R} $. The spectrum of the matrix $ M_{a, b} $ is constituted by
 \begin{align}
&\lambda_1=-\frac{1}{2} \sqrt{-b^2-2 \ii b-4 a+1}-\frac{\ii b}{2}+\frac{1}{2},
&& \lambda_2=\frac{1}{2} \sqrt{-b^2-2 \ii b-4 a+1}-\frac{\ii b}{2}+\frac{1}{2}\label{2}\\
&\lambda_3=-\frac{1}{2} \sqrt{-b^2+2 \ii b-4 a+1}+\frac{\ii b}{2}+\frac{1}{2}, &&
\lambda_4=\frac{1}{2} \sqrt{-b^2+2 \ii b-4 a+1}+\frac{\ii b}{2}+\frac{1}{2}.\label{3} 
\end{align}
 \end{lem}
 \begin{proof}
 Let $ a> 0 $ and $ b \in \mathbb{R} $. To calculate the spectrum of the matrix $ M_ {a, b} $, it is necessary to calculate the roots of its characteristic polynomial $ \mathcal {P}_{M_{a,b}} $
 \begin{align*}
\mathcal{P}_{M_{a,b}}(\lambda)&:=\begin{vmatrix}
1-\lambda & b & \sqrt{a} & 0 \\
-b & 1-\lambda & 0 &\sqrt{a}\\
-\sqrt{a} & 0 & -\lambda & 0\\
0 &-\sqrt{a} & 0 &-\lambda
\end{vmatrix}\\&= -\sqrt{a}\begin{vmatrix}
1-\lambda & \sqrt{a} & 0 \\
-b & 0 & \sqrt{a}\\
-\sqrt{a} & -\lambda & 0
\end{vmatrix} -\lambda \begin{vmatrix}
1-\lambda & b &\sqrt{a}\\
-b & 1-\lambda & 0\\
-\sqrt{a} & 0 & -\lambda
\end{vmatrix}\\
&=\lambda^{4} -2\lambda^{3}+ (2a+b^{2}+1)\lambda^{2}-2a\lambda +a^{2}.
\end{align*} 
 Now we are trying to solve
 \begin{align}\label{eq-p-1}
\mathcal{P}_{M_{a,b}}(\lambda)=\lambda^{4} -2\lambda^{3}+ (2a+b^{2}+1)\lambda^{2}-2a\lambda +a^{2}=0\,.
\end{align}
Equation \eqref{eq-p-1} is equivalent, by the change of variables
\begin{align}\label{eq-changement}
\lambda =z+\frac{1}{2}\,,
\end{align}
 to an equation without a term of degree-$ 3 $ term
 \begin{align}\label{eq-reecriture}
z^4+p\,z^2+q\,z+r=0\,,
\end{align}
where the coefficients $ p, q $ and $ r \in \mathbb{R} $ are defined as follows:
\begin{align}\label{p-q-r}
\begin{cases}
&p=\displaystyle \frac{1}{2}(2b^2+4a-1)\,,\\
&q=b^2\,,\\
&r=\displaystyle \frac{1}{16}(4b^2+16a^2-8a+1)\,.
\end{cases}
\end{align}
The principle of this method, which is called the Ferrari method (see \cite[pages 106-107]{hymers1858treatise} for more details on this method), consists in trying to factorize the first member of the equation \eqref{eq-reecriture} in the form of the product of two second-degree polynomials, in order to be able to reduce to the resolution of two second-degree equations. We notice that
\begin{align*}
z^4=(z^2+y)^2-2y\,z^2-y^2\,.
\end{align*}
The first member of equation \eqref{eq-reecriture} is then written
\begin{align*}
z^4+p\,z^2+q\,z+r&=(z^2+y)^2-2\,y\,z^2-y^2+p\,z^2+q\,z+r\\
&=(z^2+y)^2-\left( (2y-p)z^2-q\,z+y^2-r\right)\,.
\end{align*}
We will now try to determine $ y $ so that the expression in parentheses is written in the form of a square in order to be able to use the algebraic identity $\theta^2-\rho^2=(\theta -\rho)(\theta +\rho)$. The expression
$$ (2y-p) z ^ 2-qz + y ^ 2-r $$
may be considered as a second degree polynomial in $ z $. It can be put in the form of a square if its discriminant is zero. Let us calculate its discriminant:
$$\Delta=-8y^3+4py^2+8ry-4pr+q^2\,. $$
We must therefore choose $ y $ which cancels the discriminant, which amounts to solving the third degree equation in the unknown $ y $ :
$\Delta=-8y^3+4py^2+8ry-4pr+q^2=0\,.$ 
Using the equalities defined in the system \eqref{p-q-r}, we calculate the constant coefficient of the equation $ \Delta = 0 $ which is
\begin{align*}
-4pr+q^2&=-\frac{1}{8}(2b^2+4a-1)(4b^2+16a^2-8a+1)+b^4\\
&=-\frac{1}{8}\left( 2b^2+4(a-\frac{1}{4})\right)\left(4b^2+16(a-\frac{1}{4})^2\right)+b^4\\
&=-8\left(a-\frac{1}{4} \right)^3-4b^2\left(a-\frac{1}{4}\right)^2-2b^2\left( a-\frac{1}{4}\right)\\
&=\left(a-\frac{1}{4} \right)\left( -8a^2-4ab^2+4a-b^2-\frac{1}{2} \right),
\end{align*}
we notice that $ a-1/4 $ divides the constant coefficient of the equation $ \Delta = 0 $. Therefore
we verify that $ y_0 = a-1/4 $ is a solution of the previous equation because
\begin{align*}
-8\left(a-\frac{1}{4}\right)^3&+2\left(a-\frac{1}{4}\right)^2(4a+2b^2-1)
+\frac{1}{2}\left(a-\frac{1}{4}\right)(16a^2-8a+1+4b^2)\\
&-\frac{1}{8}(4a+2b^2-1)(16a^2-8a+4b^2+1)+b^4\\
&\qquad=0.
\end{align*}

Now, taking into account that $ y_0 $ cancels the discriminant ($ \Delta = 0 $), the first member of the equation \eqref{eq-reecriture} is written
\begin{align*}
z^4+p\,z^2+q\,z+r&=(z^2+y_0)^2-\left( (2y_0-p)z^2-q\,z+y_0^2-r\right)\\
&=(z^2+y_0)^2-(2y_0-p)\left(z^2-\displaystyle \frac{q}{2(2y_0-p)}\right)^2\\
&=\left(z^2+z\sqrt{2y_0-p}+y_0-\displaystyle\frac{q}{2\sqrt{2y_0-p}}\right)\\
&\times\left(z^2-z\sqrt{2y_0-p}+y_0+\displaystyle\frac{q}{2\sqrt{2y_0-p}}\right)\,,
\end{align*}
where $ \sqrt{2y_0-p} $ denotes one of the square roots of $ 2y_0-p $. The equation \eqref{eq-reecriture} is therefore equivalent to
\begin{align*}
z^2+z\sqrt{2y_0-p}+y_0-\displaystyle\frac{q}{2\sqrt{2y_0-p}}=0\,,
\end{align*}
or
\begin{align*}
z^2-z\sqrt{2y_0-p}+y_0+\displaystyle\frac{q}{2\sqrt{2y_0-p}}=0\,.
\end{align*}
The discriminants $ \Delta_+ $ and $ \Delta_- $ associated with each of the preceding equations are defined as follows:
\begin{align*}
\Delta_+&=2y_0-p-4\left(y_0-\displaystyle\frac{q}{2\sqrt{2y_0-p}}\right)=-b^2-4a-2\ii b +1\,,\\
\Delta_-&=2y_0-p-4\left(y_0+\displaystyle\frac{q}{2\sqrt{2y_0-p}}\right)=-b^2-4a+2\ii b +1\,.
\end{align*}
The solutions of the equation \eqref{eq-reecriture} are given by
\begin{align*}
z_1&=\displaystyle\frac{-\sqrt{2y_0+p}-\sqrt{\Delta_+}}{2}=\displaystyle \frac{1}{2}(-\ii b -\sqrt{A})\,,&&
z_2=\displaystyle\frac{-\sqrt{2y_0+p}+\sqrt{\Delta_+}}{2}=\displaystyle \frac{1}{2}(-\ii b +\sqrt{A})\,,\\
z_3&=\displaystyle\frac{\sqrt{2y_0+p}-\sqrt{\Delta_-}}{2}=\displaystyle \frac{1}{2}(\ii b -\overline{\sqrt{A}})\,,&&
z_4=\displaystyle\frac{\sqrt{2y_0+p}+\sqrt{\Delta_-}}{2}=\displaystyle \frac{1}{2}(\ii b +\overline{\sqrt{A}})\,,
\end{align*}
where $ A = -b^2-4a-2 \ii b + 1 $ is the auxiliary quantity which will be useful in the following and $ \overline{A} $ denotes the complex conjugate of $ A $.
From where the roots of the initial equation $ \mathcal{P}_{M_{a, b}} (\lambda) = 0 $ which characterizes the eigenvalues of the matrix $ M_{a, b} $, are given by
\begin{align*}
\lambda_1&=z_1+\displaystyle \frac{1}{2}=\displaystyle \frac{1}{2}(-\ii b -\sqrt{A}+1)&&
\lambda_2=z_2+\displaystyle \frac{1}{2}=\displaystyle \frac{1}{2}(-\ii b +\sqrt{A}+1)\\
\lambda_3&=z_3+\displaystyle \frac{1}{2}=\displaystyle \frac{1}{2}(\ii b -\overline{\sqrt{A}}+1)&&
\lambda_4=z_4+\displaystyle \frac{1}{2}=\displaystyle \frac{1}{2}(\ii b +\overline{\sqrt{A}}+1)\,.
\end{align*}
 \end{proof}
 \begin{remarque}\label{rem-symetrie}
 Note that the spectrum of the matrix $ M_{a, b} $ is independent of the sign of $ b $. This shows property $ 3) $ of Theorem \ref{thm q1}, which is physically justified by the symmetry of the system with respect to the direction of the magnetic field. In other words, this symmetry is given by
 $$ \lambda_1(a,b)=\mathcal{S}(\lambda_3(a,b))=\lambda_3(a,-b) \, \text{ and } \, \lambda_2(a,b)=\mathcal{S}(\lambda_4(a,b))=\lambda_4(a,-b)\,, $$
 where $ \mathcal{S} $ acts on functions by reflection in the second variable. We note that this symmetry on the set of $\lambda_j$ coincides with the complex conjugate.
 \end{remarque}
 Now, one is able to give the proof of Theorem \ref{thm q1}.
 \begin{proof}[Proof of Theorem \ref{thm q1}.]
 According to Lemma \ref{lem q1}, the spectrum of the matrix $ M_{a, b} $ is given by
 $$ \Sigma (M_{a,b}):=\{\lambda_1,\lambda_2, \lambda_3, \lambda_4\}, $$
 with $ \lambda_{j} $, $ j = 1, .., $ 4 are defined in \eqref{2} - \eqref{3}. Applying Corollary \ref{cor q1} to our operator $ P_{a, b} $, we obtain the spectrum of the $ P_{a, b} $ operator as follows:
 \begin{align*}
\Sigma (P_{a,b})=\{\alpha_{k}=\sum_{j=1}^{4} \, \lambda_j \, k_j, \, \forall k\in  \mathbb{N}^{4}\}.
\end{align*}
Now, we check the properties $ 1)$ and $ 2) $ of Theorem \ref{thm q1}. First, we remember that
\begin{align}\label{expression-auxiliaire1}
\lambda_1&=\displaystyle \frac{1}{2}(-\ii( b +c_2)+1-c_1), 
&&\lambda_2=\displaystyle \frac{1}{2}(-\ii (b -c_2)+1+c_1)\,,\\
\lambda_3&=\displaystyle \frac{1}{2}(\ii ( b+c_2)+1-c_1)\,,\label{expression-auxiliaire2}
&&\lambda_4=\displaystyle \frac{1}{2}(\ii (b-c_2)+1+c_1)\,.
\end{align}
where $ c = c_1 + \ii c_2 = \sqrt{A} $ and $ A = -b^2-4a-2 \ii b + 1 $ are defined in \eqref{eq_def_A} - \eqref{eq_def_c} respectively. Therefore, to calculate the asymptotics of eigenvalues when $ b \to + \infty $, it is enough to calculate the asymptotics of $ c_1 $ and $ c_2 $.
We recall that $ c = c_1 + \ii c_2 = \sqrt {A} $ is a square root of $ A = A_1 + \ii A_2 $ then
$$(c_1+\ii c_2)^2=A_1+\ii A_2\,. $$
The goal is to look for $ c_1 $ and $ c_2 $. Identifying the real part and the imaginary part in the equation above,
\begin{equation}
\begin{cases}
c_1^2-c_2^2=A_1=-b^2-4a+1\\
2c_1c_2=A_2=-2 b
\end{cases}
\end{equation}
and as $ (c_1 + \ii c_2)^2 = A $, we have the equality of absolute values
$$\vert (c_1+\ii c_2)\vert^2=\vert A_1+\ii A_2\vert \Longleftrightarrow c_1^2+c_2^2=\sqrt{A_1^2+A_2^2}=\vert A\vert\,. $$
So, we have
\begin{equation}
(c_1+\ii c_2)^2= A_1+\ii A_2 \Longleftrightarrow 
\begin{cases}\label{sys 1}
c_1^2-c_2^2=A_1\\
2\,c_1\,c_2=A_2\\
c_1^2+c_2^2=\vert A\vert
\end{cases}
\end{equation}
Solving the system of equations on the right \eqref{sys 1}, we get
\begin{equation}
\begin{cases}\label{sys 2}
c_1=\pm\sqrt{\frac{\vert A\vert +A_1}{2}}\\
c_2=\pm\sqrt{\frac{\vert A\vert -A_1}{2}}\\
2\,c_1\,c_2=A_2
\end{cases}
\end{equation}
and since $ A_2 = -2  b \leq 0 $ when $ b \geq 0 $, the real numbers $ c_1 $ and $ c_2 $ are of opposite sign. Without loss of generality we will treat the case $ b \geq 0 $, assuming that $ c_1 \geq 0 $ and therefore $ c_2 \leq 0 $.

To calculate the asymptotics of $ c_1 $ and $ c_2 $ when $ b \to + \infty $, we must first find the asymptotic of $ \vert A \vert $
\begin{align*}
\vert A\vert^2 =A_1^2+A_2^2&=(-b^2-4a+1)^2+4b^2\\
&=b^4 +b^2(8\,a+2)+(4\,a-1)^2\\
&=b^4 \left(1+\frac{2(4a+1)}{b^2}+\frac{(4a-1)^2}{b^4}\right)
\end{align*}
Then, using the Taylor-Young formula for $ \sqrt{1 + x} $ when $ x = \langle a\rangle b^{- 2} $ sufficiently small,
we obtain
\begin{align*}
\vert A\vert &=b^2 \sqrt{1+\frac{2(4a+1)}{b^2}+\frac{(4a-1)^2}{b^4}}\\
&=b^2 \left(1+\frac{(4\,a+1)}{b^2}+\frac{(4\,a-1)^2}{2b^4}-\frac{1}{8}\left(\frac{2(4a+1)}{b^2}+\frac{(4a-1)^2}{b^4}\right)^2 +\mathcal{O}(\langle a\rangle^2 b^{-6}) \right)\\
&=b^2\left( 1+\frac{(4\,a+1)}{b^2}+\frac{(4\,a-1)^2 - (4\,a+1)^2}{2b^4}+\mathcal{O}(\langle a\rangle^2 b^{-6} ) \right)\\
&=b^2+4\,a+1-\frac{8\,a}{b^2}+\mathcal{O}(\langle a\rangle^2 b^{-4} )\,.
\end{align*}
By replacing in the asymptotic $ \vert A \vert $ in equalities of system \eqref{sys 2}, we obtain
\begin{align*}
c_1=\sqrt{\frac{\vert A\vert +A_1}{2}}&=\sqrt{\frac{b^2+4\,a+1-\frac{8\,a}{b^2}+\mathcal{O}(\langle a\rangle^2 b^{-4} )-b^2+1-4\,a}{2}}\\
&=\sqrt{1-\frac{4\,a}{b^2}+\mathcal{O}(\langle a\rangle^2 b^{-4} )}\,.
\end{align*}
Then, using Taylor's formula of the function  $ x\to \sqrt{1-x} $, we get
\begin{align*}
c_1=1-\frac{2\,a}{b^2} +\mathcal{O}(\langle a\rangle^2 b^{-4} )\,.
\end{align*}
Similarly, we obtain
\begin{align*}
c_2=-\sqrt{b^2+4\,a-\frac{4\,a}{b^2}+\mathcal{O}(\langle a\rangle^2 b^{-4} )}&=-b\,\sqrt{1+\frac{4\,a}{b^2}+\mathcal{O}(\langle a\rangle^2 b^{-4} )}\\
&=-b\left(1+\frac{2\,a}{b^2}+\mathcal{O}(\langle a\rangle^2 b^{-4} )\right)    \\
&=-b-\frac{2\,a}{b} +\mathcal{O}(\langle a\rangle^2 b^{-3} )\,.
\end{align*}
Using the asymptotics of $ c_1 $ and $ c_2 $ when $ b \to + \infty $, we have
\begin{align}
1-c_1=\frac{2\,a}{b^2}+\mathcal{O}(\langle a\rangle^2 b^{-4} ) \, &\text{ and }\, 1+c_1=2-\frac{2a}{b^2}+\mathcal{O}(\langle a\rangle^2 b^{-4} )\,,\label{eg 1}\\
b-c_2=2\,b+\frac{2\,a}{b}+\mathcal{O}(\langle a\rangle^2 b^{-3} )\,&\text{ and }\, b+c_2=-\frac{2\,a}{b}+\mathcal{O}(\langle a\rangle^2 b^{-3} )\,.\label{eg 2}
\end{align}
This completes the proof of property $1)$ of Theorem
  \ref{thm q1}.
  
  Now, we will show property $ 2) $. Let $ a> 0 $ and $ b \neq 0 $, from expressions \eqref{expression-auxiliaire1}-\eqref{expression-auxiliaire2} of eigenvalues of $ M_{a, b} $ as a function of auxiliary quantities $ c_1 $ and $ c_2 $, we have that the imaginary parts of the eigenvalues are written as follows:
  $$\pm (c_2+b) \text{ and } \pm (b-c_2)\,.  $$
  Under the assumption that $ b \neq 0 $ there are two cases to treat.
  
  \textbf{$ \bullet $ Case $ b> 0 $:} In this case, according to the third equation of the system \eqref{sys 2} we have
  $$2c_1\,c_2=A_1=-2b\,.$$
  Without loss of generality, we can assume that $ c_1 \geq 0 $ and therefore $ c_2 \leq 0 $.
  So we have $ b-c_2> b> 0 $, so $ \Im (\lambda_2), \Im (\lambda_4) \neq 0 $.
  It remains to show that $ b + c_2 \neq 0 $ when $ b> 0 $, for this we will argue by contradiction. \\
    Suppose that $ b + c_2 = 0 $. Using the third equation of the system \eqref{sys 2}, we get
    \begin{align*}
   c_2+b=0\Longleftrightarrow c_2-c_1\,c_2=0
   &\Longleftrightarrow c_2(1-c_1)=0\\
   &\Longleftrightarrow c_2=0\, \text{ or }\, c_1=1\,. 
   \end{align*}
   Since $ c_2 = -b <0 $, we have $ c_2 \neq 0 $. Hence $ c_1 =1 $ , which implies that $ \lambda_1=\lambda_3 = 0 $. Now according to the first equation of system \eqref{sys 2}, we have
   $$c_1^2-c_2^2=A_2=-b^2-4a+1\Longleftrightarrow 1-b^2=-b^2-4a+1\Longleftrightarrow a=0\,, $$
   which is impossible as $ a> 0 $. Therefore, $ c_2 + b \neq 0 $, then $ \Im (\lambda_1), \Im (\lambda_3) \neq 0 $. This completes the demonstration of property $ 2) $ when $ b> 0 $.
    \end{proof}
\begin{remarque}\label{rem-distinct}
We note that if $ a> 0 $ and $ b \neq 0 $, then all the eigenvalues of $ M_{a, b} $ are distinct. Indeed, according to property $ 2) $ of Theorem \ref{thm q1} all the eigenvalues of $ M_{a, b} $ are non real {\it i.e.} $ \Im (\lambda_j) \neq 0 $, for all $ j = 1,2,3,4. $ We deduce that
$$ \lambda_1\neq \overline{\lambda}_1=\lambda_3 \,\text{ and } \,\lambda_2\neq \overline{\lambda}_2=\lambda_4.  $$
What remains is to show that $ \lambda_1 \neq \lambda_2 $ and $ \lambda_1 \neq \lambda_4 $.
We argue by contradiction. Suppose that $ \lambda_1 = \lambda_2 $. According to the equalities given in \eqref{expression-auxiliaire1} we have
$$ b+c_2=b-c_2 \,\text{ and }\, 1-c_1=1+c_1,$$
which implies that $ c_1 = c_2 = 0 $. Now, according to the second equation of the system \eqref{sys 1} we have
$ c_1 c_2 = b $ therefore $ b = 0$,  which is impossible because $ b $ was assumed to be non-zero. Hence $ \lambda_1 \neq \lambda_2 $. Similarly we can show that $ \lambda_1 \neq \lambda_4 $.
\end{remarque}   

   \section{Computation of the exponential norm of $ M_{a, b} $}
\subsection{Computation of the matrix exponential $ \ee^{- tM_{a, b}} $}
\subsubsection{Preliminary. }
The goal of this part is to introduce a method which consists in calculating the exponential of a matrix which one will use it afterwards. (see \cite[Chapter 10]{serre2001matrices} for more details on this subject).

The following proposition gives us a method to calculate the matrix exponential.
   \begin{prop}\label{prop exp}
If $P\in \mathrm{GL}_{n}(\mathbb{C})$ and $A\in  M_{n}(\mathbb{C})$, then $\ee^{P^{-1} A P}=P^{-1}\ee^A P$.
\end{prop}
We can use Proposition \ref{prop exp}, if the matrix is diagonalizable and we know the eigenvalues. This requires calculating the change-of-basis matrix, which cannot always be easily explained, especially in cases where the matrix has parameters.

There exists a simple method which is based on the interpolation of Lagrange, the only condition to be able to apply it is to be able to diagonalize the matrix and identify the eigenvalues.

Let $ A \in M_n (\Bbb C) $. By the Cayley-Hamilton Theorem and the series expansion, there exists a polynomial $Q \in \Bbb{C}[X]$ such that $Q(A) = \ee^{A}$. But for a diagonal matrix with distinct eigenvalues, we can be more explicit. Indeed, if $ Q \in \mathbb{C} [X] $ and $ A = P  D  P^{-1} $ with $ D = \mathrm{Diag} (\lambda_1, \dots, \lambda_n) $ then
$$Q(A)=P^{-1}Q(D)P.$$
It is enough to find $ Q $ such that
 $$ Q(D)=\mathrm{Diag} (Q(\lambda_1),...,Q(\lambda_n))=\mathrm{Diag}(\ee^{\lambda_1},...,\ee^{ \lambda_n})=\exp (D)\,.$$
 This is done very well by Lagrange interpolation if the eigenvalues $ \lambda_j $ of $ A $ are distinct
\begin{equation}\label{linterp}
  Q=\sum^{n}_{j=1}\,\ee^{\lambda_j} \,P_j \text{ with } P_j=\prod_{k=1,..,n, j\neq k} \frac{X-\lambda_k}{\lambda_j-\lambda_k }\,,
\end{equation}
  Finally, we deduce the expression of the exponential of $ A $ as follows:
   $$ Q(A)=P^{-1}Q(D)P=P^{-1}\,\ee^D\,P=\ee^A\,, $$
   as we wanted.
   \subsubsection{Application of Lagrange interpolation method.}
   We remember that the spectrum of the matrix $ M_{a, b} $ is given by
   $$\mathrm{Spec}(M_{a,b})=\{\lambda_1,\lambda_2,\lambda_3,\lambda_4\}\,, $$
   where the $ \lambda_j $ are defined in \eqref{2} - \eqref{3}, for all $j = 1,2,3,4 $. For $ a> 0 $ and $ b \neq 0 $, we denote $ P $ the change-of-basis matrix from the canonical basis of $ \mathbb{R}^4 $ to the basis of eigenvectors $ \mathcal{B} = \{v_1, v_2, v_3, v_4 \} $. The existence of $ P $ is justified by the fact that the eigenvalues $ \lambda_j $ are distinct (see Remark \ref{rem-distinct}) with the eigenvectors $ v_j $ defined as follows:
    \begin{align*}
  v_j =\left( \frac{\lambda_{j}^2 -\lambda_j +a}{b\sqrt{a}},\, -\frac{\lambda_j}{\sqrt{a}},\,-\frac{\lambda_j^2 -\lambda_j+a}{b\lambda_j},\, 1\right),  \,\forall j=1,..,4\,.
 \end{align*}
 To determine the matrix exponential $ \ee^{- tM_{a, b}} $, we set $ A_t = -tM_{a, b} $. According to \eqref{linterp}, there exists a polynomial $ H \in \mathbb{C} [X] $ such that
 $$\ee^{A_t}=H(A_t)\,,$$
with
 $$ H(X)=\sum_{j=1}^{4}\,\ee^{-\lambda_j t}\,P_j(X) $$
 where $$P_j(X)=\prod_{k=1
 k\neq j}^{4}\,\left(\frac{X+\lambda_k t}{-\lambda_j t +\lambda_k t}\right)\,, \forall j=1,2,3,4\,.$$
 After replacing $ X $ by $ -t  M_{a, b} $ in the polynomials
$ P_j $ for all $ j = 1, \dots, 4 $, we can see that $ P_j (-t  M_{a, b}) $ does not depend on $ t $. Using the symmetry defined in Remark \ref{rem-symetrie}, we obtain the exact expression of the matrix exponential of $ A_t $ as follows:
\begin{align}
\ee^{A_t}&=H(-tM_{a,b})\notag\\
&=2\,\left(\Re(\alpha)\,M_{a,b}^{3}+\Re (\beta)\,M_{a,b}^{2}+\Re(\gamma)\,M_{a,b}+\Re (\delta)I_{4}\right)\,,\label{forme1}
\end{align}                                               
where $ I_4 $ is the identity matrix of order $ 4 $, and the coefficients $ \alpha, \beta, \gamma, \delta \in \mathbb{C} $ are given by
\begin{align*}
\alpha&=-\left(X_1+X_2 \right)\,,\\
\beta&=\frac{1}{2}\left((c+3+ib)\,X_1+(-c+3+ib)\,X_2 \right)\,,\\
\gamma&=-\left((\vert \lambda_2\vert^2 +\lambda_3\,(c_1+1))\,X_1+(\vert \lambda_1\vert^2 +\lambda_4\,(-c_1+1))\,X_2 \right)\,,\\
\delta&=\frac{1}{4}\left(\vert \lambda_2\vert^2\,\lambda_3\, X_1+\vert \lambda_1\vert^2\,\lambda_4\,X_2 \right)\,,
\end{align*}
where $ c = c_1 + \ii c_2 = \sqrt {A} $ are the auxiliary quantities defined in \eqref{eq_def_A} and \eqref{eq_def_c},
and $ X_1 $ and $ X_2 $ take the following form :
\begin{align}
X_1&:=\frac{\ee^{-t\lambda_1}}{(\lambda_1-\lambda_2)(\lambda_1-\lambda_3)(\lambda_1-\lambda_4)}=\frac{\ee^{-t\lambda_1}}{\ii c(c_2+b)(c_1+\ii b)}\,,\label{X1}\\
X_2&:=\frac{\ee^{-t\lambda_2}}{(\lambda_2-\lambda_1)(\lambda_2-\lambda_3)(\lambda_2-\lambda_4)}=\frac{\ee^{-t\lambda_2}}{-\ii c(-c_2+b)(-c_1+\ii b)}\,.\label{X2}
\end{align}
We notice that $ X_2 $ is the symmetry of $ X_1 $ by the transformation $ \mathcal{\sigma} $
$$ \mathcal{\sigma} (f(c))=f(-c) $$
with $c=c_1+\ii c_2$, in other words, $ X_2 = \mathcal{\sigma} (X_1). $ In view of this symmetry, the expression of $ \ee^{- tM_{a, b}} $ is independent of the choice of $ \sqrt{A} $.
\begin{remarque}
We note that the expression \eqref{forme1} of the matrix exponential $ \ee^{- t M_{a, b}} $ involves the real parts of the coefficients $ \alpha, \beta, \gamma $ and $ \delta \in \mathbb{C} $ because the matrix $ \ee^{- t  M_{a, b}} $ is real. We could write the matrix $ \ee^{A_t} $ including terms $ X_3 $ and $ X_4 $ which are the complex conjugates of $ X_1 $ and $ X_2 $ respectively, since $ \Re (z) = \Re (\bar{z}) $. So, the complex conjugate symmetry helps us to simplify the expression of the matrix exponential and explains the presence of real parts in \eqref{forme1}.
\end{remarque}
We search for simple matrices which we can write the matrices $ M_{a, b}, M_{a, b}^2 $ and $ M_{a, b}^3 $ are linear combinations. For this we define
\begin{align}
&H_{vv} = \begin{pmatrix}
1&0&0&0\\
0&1&0&0\\
0&0&0&0\\
0&0&0&0
\end{pmatrix} &&H_{xx} =\begin{pmatrix}
0&0&0&0\\
0&0&0&0\\
0&0&1&0\\
0&0&0&1
\end{pmatrix} &&&
J=\begin{pmatrix}
0&0&-1&0\\
0&0&0&-1\\
1&0&0&0\\
0&1&0&0
\end{pmatrix}\label{matrix 1}\\
&J_J =\begin{pmatrix}
0&0&0&1\\
0&0&-1&0\\
0&-1&0&0\\
1&0&0&0
\end{pmatrix}&& J_{vv}=\begin{pmatrix}
0&-1&0&0\\
1&0&0&0\\
0&0&0&0\\
0&0&0&0
\end{pmatrix} &&&J_{xx}=\begin{pmatrix}
0&0&0&0\\
0&0&0&0\\
0&0&0&-1\\
0&0&1&0
\end{pmatrix}\label{matrix 2}
\end{align}
with which we can write
\begin{align}
M_{a,b}=H_{vv}-b\,J_{vv}-\sqrt{a}\,J\,,\label{exp1}
\end{align}
\begin{align}
&M_{a,b}^2=(-b^2+1)\,H_{vv}-2b\,J_{vv}-\sqrt{a}\,J+\sqrt{a}b\,J_{J}-a\,I_{4}\,,\label{exp2}
\end{align}
\begin{align}
M_{a,b}^3&=(-3b^2-a+1)\,H_{vv}+(b^3+(2a-3)b)\,J_{vv}+(\sqrt{a}b^2+\sqrt{a}(a-1))\,J\label{exp3}\\
&+2\sqrt{a}b\,J_{J}+a\,b\,J_{xx}-a\,I_4\,.\notag
\end{align} 
We note that the choice of these matrices is based on the study of the case of the quadratic operator (KFP) without magnetic field in \cite{said2018quaternionic}, where the authors chose $ H_{vv}, J_{xx}, J $ and $ H_{xx} $, and by adding the matrices $ J_J $ and $ J_{vv} $ due to the magnetic field.

Using the equalities \eqref{exp1} - \eqref{exp3}, the matrix exponential takes the following form:
\begin{align}\label{ecriture exp}
\ee^{-tM_{a,b}}=g\,H_{vv}+d\,J+e\,J_{J}+f\,J_{vv}+h\,J_{xx}+k\,(H_{xx}+H_{vv})\,,
\end{align}
where the coefficients $ g, d, e, f, h $ and $ k \in \mathbb{R} $ are given by :
\begin{align}
g&=2b(b\Re(L)-\Im(L))\,,\label{g}\\
d&=2b\sqrt{a}\Im(L)\,,\label{d}\\
e&=-2b\sqrt{a}\Re(L)\,,\label{e}\\
f&=2b(\Re(L)+b\Im(L))-2ba\Re(O)\,,\label{f}\\
h&=-2ab\Re(O)\,,\label{h} \\
k&=2ab\Im(O)\,,\label{k}
\end{align}
where 
\begin{align}\label{def-L_O}
 L=\lambda_1 X_1 +\lambda_2 X_2 \text{ and } O=X_1+X_2\,,
\end{align} 
and $ X_1 $ and $ X_2 $ are defined in \eqref{X1} and \eqref{X2} respectively.
\subsection{Proof of Theorem \ref{thm_2,4}}
\subsubsection{Intermediate case. }
In this part, we seek to find the eigenvalues for a specific class of real matrices, which are written as a linear combination of matrices $ H_{vv} - H_{xx}, J_{J}, $ and $$K= [J,H_{vv}]=\begin{pmatrix}
0&0&1&0\\
0&0&0&1\\
1&0&0&0\\
0&1&0&0
\end{pmatrix}$$ where $ H_{vv}, H_{xx}, J $ and $ J_J $ are defined in \eqref{matrix 1} and \eqref{matrix 2}.
\begin{prop}\label{prop-intermediaire}
Let $A\in M_4 (\Bbb R)$ be such that 
$$A=\beta_1 (H_{vv}-H_{xx})+\beta_2 K+\beta_3 J_{J}, $$
where $\beta_j\in \mathbb{C}$ for all $j=1,2,3$. Then the eigenvalues of the matrix $ A $ are of the following form:\begin{align*}
\alpha_1=-\sqrt{ \beta_{1}^2+ \beta_2^2+\beta_3^2}\,\text{ and } \,
\alpha_2=\sqrt{ \beta_{1}^2+ \beta_2^2+\beta_3^2}
\end{align*}
\end{prop}
\begin{proof}
We seek to find the spectrum of the matrix $ A $, and we will therefore find the roots of its characteristic polynomial $ \mathcal{P}_A $
\begin{align*}
\mathcal{P}_A (X):&=\begin{vmatrix}
\beta_1-X&0&\beta_3&\beta_2\\
0&\beta_1-X&-\beta_2&\beta_3\\
\beta_3&-\beta_2&-\beta_1-X&0\\
\beta_2&\beta_3&0&-\beta_1-X
\end{vmatrix} \\
&=(\beta_1-X)\begin{vmatrix}
\beta_1-X&-\beta_2&\beta_3\\
-\beta_2&-\beta_1-X&0\\
\beta_3&0&-\beta_1-X
\end{vmatrix}+\beta_3\begin{vmatrix}
0&\beta_1-X&\beta_3\\
\beta_3&-\beta_2&0\\
\beta_2&\beta_3&-\beta_1-X
\end{vmatrix}\\
&\qquad \qquad \qquad-\beta_2\begin{vmatrix}
0&\beta_1-X&-\beta_2\\
\beta_3&-\beta_2&-\beta_1-X\\
\beta_2&\beta_3&0
\end{vmatrix}\\
&=X^4-2(\beta_1^2+\beta_2^2+\beta_3^2)\,X^2+(\beta_1^2+\beta_2^2+\beta_3^2)^2.
\end{align*}
Now we solve the following equation:
$$ \mathcal{P}_A(X)=X^4-2(\beta_1^2+\beta_2^2+\beta_3^2)\,X^2+(\beta_1^2+\beta_2^2+\beta_3^2)^2=0.$$
Notice that the polynomial $ \mathcal{P}_A (X) $ is a square which can be factorized as follows:
\begin{align*}
\mathcal{P}_A(X)&=(X^2-\left(\beta_1^2+\beta_2^2+\beta_3^2)\right)^2\\
&=\left(X-\sqrt{\beta_1^2+\beta_2^2+\beta_3^2}\right)^2\,\left(X+\sqrt{\beta_1^2+\beta_2^2+\beta_3^2}\right)^2.
\end{align*}
Then the polynomial $ \mathcal{P}_A (X) $ admits two double roots $ \alpha_1 $ and $ \alpha_2 $
\begin{align}\label{spectre_A}
\alpha_1=-\sqrt{ \beta_{1}^2+ \beta_2^2+\beta_3^2}\,\text{ and } \,
\alpha_2=\sqrt{ \beta_{1}^2+ \beta_2^2+\beta_3^2},
\end{align}
which constitute the spectrum of $ A $.
\end{proof}
\begin{remarque}\label{remark crucial}
Let $ B $ be the matrix
$$ B=\beta_1\, (H_{vv}-H_{xx})+\beta_2 \,K+\beta_3\,J_J +\beta_4\,I_4.$$
Then the spectrum of this matrix is made up of two eigenvalues of multiplicity $2$ \begin{align*}
\gamma_1=-\sqrt{ \beta_{1}^2+ \beta_2^2+\beta_3^2}+\beta_4\,\text{ and } \,
\gamma_2=\sqrt{ \beta_{1}^2+ \beta_2^2+\beta_3^2}+\beta_4.
\end{align*}
Indeed, the spectrum of a matrix satisfies the following property:
$$ \mathrm{Spec}(B)=\mathrm{Spec}(A+\beta_4\,I_4)=\mathrm{Spec}(A)+\beta_4\,.$$
Then, according to Proposition \ref{prop-intermediaire} 
$$ \mathrm{Spec}(A)=\{\alpha_1, \alpha_2\}\,,$$ where $\alpha_1$ and $\alpha_2$ are defined in \eqref{spectre_A}. Therefore, the spectrum of $ B $ is
$\mathrm{Spec}(B)=\{\alpha_1+\beta_4, \alpha_2+\beta_4\}.$ 
\end{remarque}

\subsubsection{Computation of the norm of the matrix $ \ee^{-tM_ {a, b}} $. }
In this section, we will describe the matrix exponential $ \ee^{- tM_ {a, b}} $ in terms of the parameters $ a $ and $ b $ for any $ t \geq 0 $.
\begin{prop}\label{prop norme}
Let $ a> 0 $ and $ b \in \mathbb{R} $.
Then
 $$\Vert \ee^{-tM_{a,b}}\Vert=\sqrt{\left(\sqrt{P_1^2+P_2^2+(P_3/2)^2} +P_4/2\right)}\,, $$
 where the $ P_j $ are
  \begin{align}
 P_1&=-4\,b^2\,\sqrt{a}\,\vert L\vert^{2}\,,\label{p1}\\
 P_2&=b\,P_1-8\,b^2\,a\,\sqrt{a}\,\Im(O\,\bar{L})\,, \label{p2}\\
 P_3&=4\,b^2\,(b^2+1)\vert L\vert^{2}+8\,a\,b^2\,(b\,\Im(O\,\bar{L})-\Re(O\,\bar{L})),\label{p3}\\
 P_4&=P_3+8\,a\,b^2\,(\vert L\vert^2 +a\,\vert O\vert^2 )\,.\label{p4}
 \end{align}
 where $ L $ and $ O $ are defined in the equality \eqref{def-L_O}.
 \end{prop}
 \begin{proof}
Let $ a> 0 $ and $ b \in \mathbb{R} $. According to section 4.2, the matrix $ \ee^{- tM_{a, b}} $ can be written
 $$ \ee^{-tM_{a,b}}=g\,H_{vv}+d\,J+e\,J_{J}+f\,J_{vv}+h\,J_{xx}+k\,(H_{xx}+H_{vv})\,, $$
 where the coefficients $ g, d, e, f, h $ and $ k $ are defined in the equalities \eqref{g}--\eqref{k}. We recall that the norm of the matrix exponential can be calculated as follows:
 \begin{align} \label{norme 2}
\Vert \ee^{-tM_{a,b}}\Vert = \sqrt{\rho (\ee^{-tM_{a,b}}\,\ee^{-tM^{*}_{a,b}})}\,,
\end{align}
where $ \rho (M) $ denotes the largest eigenvalue of $ M $ in absolute value, for $ M \in \mathbb{M}_{n} (\mathbb{C}) $.
It remains for us to calculate the matrix $ \ee^{- tM_{a, b}}  \ee^{- tM^{*}_{a, b}} $. Using equality \eqref{ecriture exp}, we get
\begin{align*}
\ee^{-tM_{a,b}}\,\ee^{-tM^{*}_{a,b}}&=\frac{1}{2}(g^2+2kg+f^2-h^2)(H_{vv}-H_{xx})\\
&+(dg+e(f-h))\,K
+(eg+2ek-df-hd)J_{J}\\
&+\frac{1}{2}(g^2+2\,kg+f^2-h^2+2(k^2+d^2+e^2))(H_{vv}+H_{xx}).
\end{align*}
According to the preceding equality, we obtain that the matrix $ \ee^{- tM_{a, b}} \ee^{- tM_{a, b}^*} $ belongs to the class of matrices introduced in the previous section. According to the intermediate case and Remark \ref{remark crucial}, we have that the greatest eigenvalue is
\begin{align}\label{rho}
\rho (\ee^{-tM_{a,b}}\,\ee^{-tM_{a,b}^*})= \sqrt{\left(\sqrt{P_1^2+P_2^2+(P_3/2)^2} +P_4/2\right)}\,,
\end{align}
where the $ P_j $ are
\begin{align*}
P_1&=dg+e(f-h)\,,\\
P_2&=eg+2ek-df-hd\,,\\
P_3&=g^2+2kg+f^2-h^2\,,\\
P_4&=g^2+2kg+f^2-h^2+2(k^2+d^2+e^2)\,.
\end{align*}
To complete the proof, we use the equalities \eqref{g}--\eqref{k} to obtain the equalities \eqref{p1}--\eqref{p4}\,.
 \end{proof}
Now, we will give the proof of Theorem \ref{thm_2,4}.
\begin{proof}[Proof of Theorem \ref{thm_2,4}]
According to the previous Proposition, we have
\begin{align}\label{formula-recall}
\Vert \ee^{-tM_{a,b}}\Vert=\sqrt{\left(\sqrt{P_1^2+P_2^2+(P_3/2)^2} +P_4/2\right)}\,.
\end{align}
To complete the proof of Theorem \ref{thm_2,4}, we must express the coefficients $ P_1, P_2, P_3 $ and $ P_4 $ in terms of the auxiliary quantities $ A $ and $ c $ defined in \eqref{eq_def_A}-\eqref{eq_def_c}. As the coefficients indicated just before are calculated as functions of $ O $ and $ L $, we are first interested in explicitly calculating these quantities.
We notice that
$$(\lambda_1 - \lambda_3)(\lambda_1 - \lambda_4) = -2\ii b\lambda_1 \,\text{ and }\, \lambda_1\,\lambda_2=a\,. $$
This allows us to write
\begin{align*}
	X_1 := \frac{\ee^{-t\lambda_1}}{(\lambda_1-\lambda_2)(\lambda_1-\lambda_3)(\lambda_1-\lambda_4)}
	= -\frac{\ee^{-t\lambda_1}}{2\ii bc\lambda_1}
	= -\frac{\lambda_2\,\ee^{-t\lambda_1}}{2\ii abc},
\end{align*}
where $ c = \sqrt{A} $ is the auxiliary quantity defined in \eqref{eq_def_c}.
This leads us to simplify the expressions of $ O $ and $ L $ as follows:
\begin{align*}
	O &= X_1 + X_2 = \frac{1}{2\ii abc}(\lambda_1\ee^{-t\lambda_2} - \lambda_2\ee^{-t\lambda_1}),\\
	L &= \lambda_1 X_1 + \lambda_2 X_2 = \frac{1}{2\ii bc}(\ee^{-t\lambda_2} - \ee^{-t\lambda_1}).
\end{align*}
We introduce temporary notations similar to $ \cosh $ and $ \sinh $
\begin{align}\label{def_SH-CH}
\CH = \frac{1}{2}(\ee^{-t\lambda_1} + \ee^{-t\lambda_2}), \quad \SH = \frac{1}{2}(\ee^{-t\lambda_1}-\ee^{-t\lambda_2}),
\end{align}
with which we write
\[
	O = -\frac{1}{2\ii abc}(c\CH + (1-\ii b)\SH), \quad L = -\frac{1}{\ii bc}\SH.
\]
This allows us to express the quantity $ O  \overline{L} $ as follows
$$O\overline{L}= \frac{1}{2ab^2|A|}\left(c\CH\overline{\SH} + (1-\ii b)|\SH|^2\right).$$
Now, we are able to explain the coefficients that determine the norm of the matrix exponential. We start by calculating $ P_1 $:
\[
	P_1 = -4\,b^2\,\sqrt{a}\,\vert L\vert^{2}=-\frac{4\sqrt{a}}{|A|}|\SH|^2.
\]
Then, we can calculate $ P_2 $:
\begin{align*}
	P_2 &= bP_1 - 8b^2a^{3/2} \Im(O\overline{L})
	\\ &= -\frac{4b\sqrt{a}}{|A|}|\SH|^2 - \frac{8b^2a^{3/2}}{2ab^2|A|}\Im\left(c \CH \overline{\SH} + (1-\ii b)|\SH|^2\right)
	\\ &= -\frac{4\sqrt{a}}{|A|}\Im(c\CH\overline{\SH}).
\end{align*}
Then, we calculate $ P_3 $ as follows:
\begin{align*}
	P_3 &= 4b^2(b^2+1)|L|^2 + 8ab^2(b\Im(O\overline{L}) - \Re(O\overline{L}))
	\\ &= 4b^2(b^2+1)\frac{1}{b^2|A|}|\SH|^2 + \frac{8ab^2}{2ab^2|A|}(b\Im(c\CH\overline{\SH} + (1-\ii b)|\SH|^2)\\& - \Re(c\CH\overline{\SH} + (1-\ii b)|\SH|^2)
	\\ &= \frac{4}{|A|}(b^2+1)|\SH|^2 + \frac{4}{|A|}(b\Im c\CH\overline{\SH} - b^2|\SH|^2 - \Re c\CH\overline{\SH} - |\SH|^2)
	\\ &= \frac{4}{|A|}(b\Im (c\CH\overline{\SH}) - \Re (c\CH\overline{\SH})).
\end{align*}
Finally, we analyze $P_4$, for which we need to study
\begin{align*}
	8a^2b^2 |O|^2 &= \frac{2}{|A|}|c\CH + (1-\ii b)\SH|^2
	\\ &= \frac{2}{|A|}\left(|A|\,|\CH|^2 + 2\Re(c\CH(1+\ii b)\overline{\SH}) + (1+b^2)|\SH|^2\right)
	\\ &= 2|\CH|^2 + \frac{4}{|A|}(\Re (c\CH\overline{\SH}) - b\Im (c\CH\overline{\SH})) + \frac{2}{|A|}(1+b^2)|\SH|^2
	\\ &= 2|\CH|^2 + \frac{2}{|A|}(1+b^2)|\SH|^2 - P_3.
\end{align*}
We can conclude that
\begin{align*}
	P_4 &= P_3 + 8ab^2(|L|^2 + a|O|^2)
	\\ &= P_3 + \frac{8a}{|A|}|\SH|^2 + 2|\CH|^2 + \frac{2}{|A|}(1+b^2)|\SH|^2 - P_3
	\\ &= \frac{2}{|A|}(1+b^2+4a)|\SH|^2 + 2|\CH|^2.
\end{align*}
Inserting into formula \eqref{formula-recall}, we obtain
\begin{align}\label{formule_temporaire}
	\|\ee^{-tM_{a,b}}\|^2 &= |\CH|^2 + \frac{1}{|A|}(1+b^2+4a)|\SH|^2 \notag\\&+ \frac{1}{|A|}\sqrt{16a(|\SH|^4 + (\Im (c\CH\overline{\SH}))^2) + 4(b\Im (c\CH\overline{\SH}) - \Re (c\CH\overline{\SH}))^2}.
\end{align}
To end the demonstration, using the expressions given in \eqref{def_SH-CH}, we get
\begin{align}
\CH\overline{\SH}&=\frac{\ee^{-t}}{2}\,(\sinh (c_1t)-\ii \,\sin (c_2t)),\\
\Im (c\CH\overline{\SH})&=\frac{\ee^{-t}}{2}\,(c_2\sinh (c_1t)-c_1\sin (c_2t)),\\
\Re (c\CH\overline{\SH})&=\frac{\ee^{-t}}{2}\,(c_1\sinh (c_1t)+c_2\sin (c_2t)),\\
\vert \SH\vert^2&=\frac{\ee^{-t}}{2} (\cosh (c_1t)-\cos (c_2t)),\\
\vert \CH\vert^2&=\frac{\ee^{-t}}{2} (\cosh (c_1t)+\cos (c_2t)).
\end{align}
Finally, we simplify the expression of the exponential norm in the following form
$$\|\ee^{-tM_{a,b}}\|^2 = \frac{1}{|A|}\ee^{-t}\left(T + \sqrt{S}\right).$$
Using the expression of the norm obtained in \eqref{formule_temporaire}, the previous equalities and $ 1 + b^2 + 4a = 2-A_1 $, we have
\begin{align*}
T&:=\ee^{t}\,\vert A\vert\,(\vert\CH|^2 + \frac{1}{|A|}(1+b^2+4a)|\SH|^2)\\
&=\frac{1}{2}\left((|A| -A_1+2)\cosh(c_1 t) + (|A|+A_1-2)\cos(c_2 t)\right),
\end{align*}
where $ A_1 = \Re (A) $ is defined in \eqref{eq_def_A}. Similarly, we have
\begin{align*}
S&:=\ee^{2t}\left(16a(|\SH|^4 + (\Im (c\CH\overline{\SH}))^2) + 4(b\Im (c\CH\overline{\SH}) - \Re (c\CH\overline{\SH}))^2\right)\\
&=4a(\cosh(c_1t)-\cos (c_2 t))^2+(4ac_2^2+(bc_2-c_1)^2)\sinh^2(c_1t)\\
&+(4ac_1^2+(bc_1+c_2)^2)\sin^2(c_2t)\\
&+(-8ac_1c_2-2(bc_1+c_2)(bc_2-c_1))\sin(tc_2)\sinh (c_1t).
\end{align*}
In order to complete the proof of the theorem, we will calculate each coefficient of the previous equality. We start by calculating
\begin{align*}
4ac_2^2+(bc_2-c_1)^2&=(4a+b^2)c_2^2+c_1^2-2bc_1c_2\\
&=(4a+b^2)(\frac{\vert A\vert -A_1}{2})+\frac{\vert A\vert +A_1}{2}+2b^2\\
&=(1+b^2+4a)\frac{\vert A\vert}{2}+(1-4a-b^2)\frac{A_1}{2}+2b^2\\&=(1+b^2+4a)\frac{\vert A\vert}{2}+\frac{A_1^2}{2}+\frac{A_2^2}{2}\\
&=(2-A_1)\frac{\vert A\vert}{2}+\frac{\vert A\vert^2}{2},
\end{align*}
we used the equations of the system \eqref{sys 2} and the fact that $1+b^2+4a=2-A_1$, $A_2=-2b$ and $A_1=1-b^2-4a$. Similarly, we can calculate
\begin{align*}
4ac_1^2+(bc_1+c_2)^2=(2-A_1)\frac{\vert A\vert}{2}-\frac{\vert A\vert^2}{2}.
\end{align*} 
The last coefficient takes the form
\begin{align*}
-8ac_1c_2-2(bc_1+c_2)(bc_2-c_1)&=2c_1c_2(-4a-b^2+1)+2b(c_1^2-c^2_2)\\
&=2c_1c_2\,A_1+2b\,A_1\\
&=(2c_1c_2+2b)\,A_1\\&=0,
\end{align*}
where we also used the equations of the system \eqref{sys 2}.
Consequently, we can rewrite $ S $ in the following form:
\begin{align*}
S &= 4a(\cosh (c_1t) - \cos (c_2 t))^2 + \frac{1}{2}(2-A_1)|A|(\sinh^2 (c_1 t) + \sin^2 (c_2 t)) 
	\\ & \qquad + \frac{1}{2}|A|^2(\sinh^2 (c_1 t) - \sin^2 (c_2 t))
\end{align*}
Then, by developing the functions $\cosh^2(c_1t), \cos^2(c_2t), \sinh^2(c_1t)$ and $\sin^2(c_2t)$ in terms of
$\cosh (2c_1t)$ and $\cos(2c_2t)$, we obtain
\begin{align*}
S&= 8a(1-\Re(\cosh ct)) + \frac{1}{4}R|A|(\cosh 2c_1t - \cos 2c_2t)
	\\ & \qquad + (2a + \frac{1}{4}|A|^2)(\cosh 2c_1 t + \cos 2c_2t - 2),
\end{align*}
where $ \Re (\cosh ct) = \cosh (c_1t) \cos (c_2t) $ and we recall that $ c = c_1 + \ii c_2$ defined in \eqref{eq_def_c}. This completes the proof of Theorem \ref{thm_2,4}.
\end{proof}
\begin{remarque}
We note that the Lie algebra which we used to decompose the matrix exponential $ \ee^{- tM_{a, b}} $ coincides with the Lie algebra that we found for the operator $ P_{a, b} $ (see \cite[Section 3.1]{zk2} for more details), in the sense that the operator $ P_{a, b} $ is considered as a vector field polynomial on a stratified and type-2 Lie algebra. Indeed, let
$$ \mathcal{G}_1=\mathrm{Vect}\{H_{vv}, H_{xx}, J_{vv}, J_{xx}\}\, \text{ and }\, \mathcal{G}_2=\mathrm{Vect}\{J,J_J\}\,,$$
where the matrices mentioned above are defined in \eqref{matrix 1} and \eqref{matrix 2}, and define the subalgebra
$\mathcal{G}_3=\mathrm{Vect}\{K,K_J\}\,,$
with $$K=[J,H_{vv}]=\begin{pmatrix}
0&0&1&0\\
0&0&0&1\\
1&0&0&0\\
0&1&0&0
\end{pmatrix} \,\text{ and }\,K_J=[J_J, H_{vv}]=\begin{pmatrix}
0&0&0&-1\\
0&0&1&0\\
0&-1&0&0\\
1&0&0&0
\end{pmatrix}.$$
These matrices generate the Lie algebra $ \mathcal{G} = \mathcal{G}_1 \oplus \mathcal{G}_2 \oplus \mathcal{G}_3 $ of dimension $ 8 $ used in \cite{zk2}.
\end{remarque}
\section{Precise estimates in several regimes of the exponential norm}
\subsection{Asymptotic expansion as $t \to 0^+$}
In this part, we will find a complete asymptotic expansion for the exponential norm of $ -tM_{a, b} $ when $ t \to 0^+ $. We will expand the quantities $ T $ and $ S $ when $ t \to 0^+ $ which define the norm of the matrix $ \ee^{- tM_{a, b}} $.
\begin{prop}
\label{prop_developpement}
The quantities $T$ and $S$ in \eqref{eq_def_T} and \eqref{eq_def_S} can be expressed as
\[
	T = \sum_{k = 0}^\infty |A|\tau_k t^{2k}, \quad S = \sum_{k=1}^\infty |A|^2\sigma_k t^{2k},
\]
where each $\tau_k$ and $\sigma_k$ can be expressed as a polynomial in $A_1 = 1-b^2-4a$ and $|A|^2 = (1-b^2-4a)^2 + 4b^2$ (and therefore as a polynomial in $a$ and $b^2$).
\end{prop}
\begin{proof}
To express the quantities $T$ and $S$, we introduce the notation
\begin{align*}
	B_{+,k} &= (c_1^2)^k + (-1)^k(c_2^2)^k = 2^{-k}\left((A_1 + |A|)^k + (A_1 - |A|)^k\right) \\
	B_{-,k} &= (c_1^2)^k - (-1)^k(c_2^2)^k = 2^{-k}\left((A_1 + |A|)^k - (A_1 - |A|)^k\right).
\end{align*}
We start by expressing $ T $ using the series expansion of the functions $ \cosh (c_1t) $ and $ \cos (c_2t) $:
\begin{align*}
T&=\frac{1}{2}\left((|A| -A_1+2)\cosh(c_1 t) + (|A|+A_1-2)\cos(c_2 t)\right)\\
&=\frac{1}{2}(|A| -A_1+2)\sum_{k=0}^{+\infty}\frac{(c_1t)^{2k}}{(2k)!}+\frac{1}{2}(|A| +A_1-2)\sum_{k=0}^{+\infty} (-1)^k\,\frac{(c_1t)^{2k}}{(2k)!}\\
&=\frac{1}{2}\sum_{k=0}^{+\infty}(\vert A\vert B_{+,k}+(2-A_1)B_{-,k})\frac{t^{2k}}{(2k)!}.
\end{align*}
By comparison, we then obtain
\[
	(2k)!|A|\tau_k = \frac{1}{2}\left(|A|B_{+,k} + (2-A_1)B_{-,k}\right)\quad\forall k\in \mathbb{N}.
\]
Similarly we can express the coefficients of $ S $ for all $ k \in \mathbb{N}^* $:
\begin{align*}
	(2k)!|A|^2\sigma_k &= -8a\Re(A^k) + 2^{2(k-1)}\left((2-A_1)|A|B_{-, k} + (|A|^2 + 8a)B_{+, k}\right)
	\\ &= 8a(2^{2(k-1)}B_{+, k} - \Re(A^k)) + 2^{2(k-1)}\left((2-A_1)|A|B_{-, k} + |A|^2 B_{+, k}\right),
\end{align*}
and $\sigma_0=0$ because $1-\Re(\cosh 0) = \cosh 0 - \cos 0 = \cosh 0 + \cos 0 - 2 = 0$.

Now, we will calculate $ \Re (A^k), B_{+, k} $ and $ B_{-, k} $, for all $ k \in \Bbb N $. 
By the binomial formula,
\begin{align*}
\Re (A^k)=&\Re ((A_1+\ii A_2)^k)\\
=&\Re \left( \sum_{j=0}^{k}\binom{k}{j}A_1^{k-j}(\ii A_2)^j\right)\\
=&\sum_{j=0}^{k}\binom{k}{j}A_1^{k-j}\Re ((\ii A_2)^j)\\
=&\sum_{j=0}^{\lfloor k/2\rfloor} \binom{k}{2j} (-1)^j A_1^{k-2j} A_2^{2j}
\end{align*}
Similarly, we can express $ B_{+, k} $ and $ B_{-, k} $
\begin{align*}
B_{+, k} &= 2^{1-k}\sum_{j = 0}^{\lfloor k/2\rfloor} \binom{k}{2j}A_1^{k-2j}|A|^{2j}, \\
	B_{-,k} &= 2^{1-k}\sum_{j = 0}^{\lfloor (k-1)/2\rfloor} \binom{k}{2j+1}A_1^{k-2j-1}|A|^{2j+1}.
\end{align*}
Every term in $|A|\tau_k$ is divisible by $|A|$ leaving a polynomial in $A_1$ and $|A|^2$. As for $|A|^2\sigma_k$, the latter terms with $|A|B_{-, k}$ and $|A|^2B_{+,k}$ are divisible by $|A|^2$ leaving polynomials in $A_1$ and $|A|^2$. To complete the proof we must study the first terms of $ \vert A \vert^2 \sigma_k $. Since $A_2^2 = |A|^2 - A_1^2$, we can write
\begin{align*}
	\Re(A^k) &= \sum_{j=0}^{\lfloor k/2 \rfloor} \binom{k}{2j}(-1)^j A_1^{k-2j}(|A|^2 - A_1^2)^{j}
	\\ &= \sum_{j=0}^{\lfloor k/2 \rfloor}\sum_{\ell = 0}^{j} \binom{k}{2j}\binom{j}{\ell}(-1)^\ell A_1^{k-2j}(A_1^2)^{j-\ell }(|A|^2)^\ell.
\end{align*}
The coefficient of $A_1^k$, given by the terms where $\ell = 0$, is
\[
	\sum_{j = 0}^{\lfloor k/2\rfloor} \binom{k}{2j} = \frac{1}{2}\left((1+1)^k - (1-1)^k\right) = 2^{k-1},
\]
which coincides exactly with the coefficient of $A_1^k$ in $2^{2(k-1)}B_{+, k}$. These terms cancel, allowing us to conclude that each $|A|^2\sigma_k$ is indeed divisible by $|A|^2$ leaving a polynomial in $A_1$ and $|A|^2$, completing the proof of the proposition.
\end{proof}
\begin{remarque}\label{rem-generale}
Since $ \vert A_1 \vert $ is less than or equal to $ \vert A \vert $ and $$a=(1-b^2-A_1)/4\leq \langle A\rangle/4, $$ it is obvious that for any $N \in \Bbb{N}$ and $R > 0$, there exists $C = C(N, R) > 0$ such that when $\langle A\rangle t^2 \leq R$,
\[
	\left|T - \sum_{k = 0}^N |A|\tau_k t^{2k}\right| \leq C\langle	A\rangle^{N+2} |t|^{2N+2}
\]
and
\[
 \left|S - \sum_{k=1}^N | A|^2\sigma_k t^{2k}\right|\leq  C \langle A\rangle^{N+3} |t|^{2N+2}.
\]
 Since one can compute that $\sigma_1 = \tau_0 = 1$, we obtain (when $\langle A\rangle t^2$ is bounded)
\begin{align*}
	\ee^{t}\|\ee^{-tM_{a,b}}\|^2 & = 1 + \sum_{k = 1}^N \tau_k t^{2k} + \BigO(\langle A\rangle^{N+1}t^{2N+2}) 
	\\ & \qquad + t\sqrt{1 + \sum_{k=1}^{N-1} \sigma_{k+1}t^{2k} + \BigO(\langle A\rangle^{N+1}t^{2N})}.
\end{align*}
For $\langle A\rangle t^2$ sufficiently small, we can expand the square root in a Taylor series and multiply by the Taylor series for $\ee^{-t}$ to obtain a complete asymptotic expansion for $\|\ee^{-tM_{a,b}}\|^2$ (and therefore for $\|\ee^{-tM_{a,b}}\|$ if one takes another square root).
\end{remarque}
To end this part, we will give the demonstration of Proposition \ref{prop_Taylor}.
\begin{proof}[Proof of Proposition \ref{prop_Taylor}]
We notice that in Proposition \ref{prop_Taylor}, we seek the asymptotic development of the norm $ \| \ee^{- tM_{a, b}} \| $ up to the order $ 6 $. For that and according to the Remark \ref{rem-generale} it suffices to calculate the terms $ \tau_k $ and $ \sigma_k $ to the order $ 3 $.
We will start by calculating the coefficients $ B_{+, k} $ and $ B_{-, k} $ following the equalities in Proposition \ref{prop_developpement}. By simple calculations we can show that
\begin{align*}
	\{B_{+, k}\}_{k=0}^3 &= \{2, A_1, \frac{1}{2}(A_1^2 + |A|^2), \frac{1}{4}(A_1^3 + 3A_1|A|^2)\}
	\\
	\{B_{-, k}\}_{k=0}^3 &= \{0, |A|, A_1|A|, \frac{1}{4}(3A_1^2|A| + |A|^3)\}.
\end{align*}
Using the values of $ B_{+, k} $ and $ B_{-, k} $ with $ k = 0, 1, 2, 3$, we can calculate
\begin{align*}
	\tau_0 &= \frac{1}{2|A|}(2|A| + 0) = 1, \\
	\tau_1 &= \frac{1}{4|A|}(|A|A_1 + (2-A_1)|A|) = \frac{1}{2}, \\
	\tau_2 &= \frac{1}{48|A|}\left(\frac{1}{2}(A_1^2 + |A|^2) + (2-A_1)A_1|A|\right) = \frac{1}{24}(1-4a)
\end{align*}
And $ \tau_3 $ is 
\begin{align*}
	\tau_3 &=\frac{1}{6!}\frac{1}{2\vert A\vert}(\vert A\vert B_{+,3}+(2-A_1)B_{-,3})\\
	&=\frac{1}{5760}(2A_1(\vert A\vert^2-A_1^2)+6A_1^2+2\vert A\vert^2)\\
	&=\frac{1}{5760}(8A_1b^2+6A_1^2+2A_1^2+2A_2^2)\\
	&=\frac{1}{720}(A_1(b^2+A_1)+b^2)\\
	&=\frac{1}{720}(1-4a(2-4a-b^2)).
\end{align*}
We continue by calculating
\begin{align*}
\sigma_1 &= \frac{1}{2|A|^2}\left(-8a A_1 + (2-A_1)|A|^2 + (|A|^2 + 8a)A_1\right) = 1
\end{align*}
and $ \sigma_2$
\begin{align*}
\sigma_2 &=\frac{1}{24\vert A\vert^2}\left(8a\Re (A^2)+4((2-A_1)\vert A\vert B_{-,2}+(\vert A\vert^2+8a)B_{+,k})\right)\\
&=\frac{1}{24\vert A\vert^2}\left(8a(-A_1^2+A_2^2+2A_1^2+2\vert A\vert^2)\right)\\
&+\frac{1}{24\vert A\vert^2}\,4\left( (2-A_1)A_1\vert A\vert^2+\frac{1}{2}(A_1^2+\vert A\vert^2) \right)\\
&=\frac{1}{24}\left( 24a+4(2A_1+\frac{1}{2}(\vert A\vert^2-A_1^2)) \right)\\
&=\frac{1}{24}\left(24a+8(1-4a) \right)\\
&=\frac{1}{3}(1-a).
\end{align*}
(Even the computation for $ \sigma_3= (4a^2 - 17a + 4 + ab^2)/90 $  becomes somewhat long.)
According to Remark \ref{rem-generale}, we can expand $ \Vert \ee^{- tM_{a, b}} \Vert $ when $ t \to 0^+ $ to order $ 6 $
\begin{align*}
	\|\ee^{-tM_{a,b}}\|^2 &=\ee^{-t}\left(  1 + \tau_1t^2 + \tau_2 t^4 + \tau_3 t^6 + \BigO(\langle A\rangle^4 \,t^8) + t\sqrt{1 + \sigma_2 t^2 + \sigma_3 t^4 + \BigO(\langle A\rangle^4\,t^6)}\right)
	\\ &=\ee^{-t}\left( 1 + t + \tau_1 t^2 + \frac{1}{2}\sigma_2 t^3 + \tau_2 t^4 + \left(\frac{1}{2}\sigma_3 - \frac{1}{8}\sigma_2^2\right)t^5 +  \tau_3 t^6 + \BigO(\langle A\rangle^4\, t^7)\right).
\end{align*}
We multiply by the Taylor series for $ \ee^{- t} $ to order $ 6 $. The coefficient of $ t ^ 0 $ in $ \| \ee^{- tM_{a, b}} \|^2 $ is therefore $ 1 $, while the coefficient of $ t $ is $ 1-1 = 0 $. The coefficient of $ t ^ 2 $ is given by
\[
	\frac{1}{2}\tau_1 - 1 + \frac{1}{2} = 0,
\]
and the coefficient of $t^3$ is
\[
	\frac{1}{2}\sigma_2 - \tau_1 + \frac{1}{2} - \frac{1}{6} = \frac{1}{6}(1-a) - \frac{1}{6} = -\frac{a}{6}.
\]
The coefficient of $ t ^ 4 $ is
\[
\frac{1}{24}+\tau_2+\frac{1}{2}\tau_1-\frac{\sigma_2}{2}-\frac{1}{6}=\frac{1}{12}+\frac{1}{4}-\frac{1}{3}+\frac{a}{6}-\frac{a}{6}=0.
\] 
Next, the coefficient of $ t ^ 5 $ is
\begin{align*}
&\frac{\sigma_3}{2}-\frac{\sigma_2^2}{8}-\frac{1}{120}-\tau_2+\frac{1}{24}+\frac{\sigma_2}{4}-\frac{\tau_1}{6}\\
&=\left( \frac{4}{180}-\frac{1}{72}\right)\,a^2+\frac{1}{180}\,ab^2+\left( -\frac{17}{180}+\frac{2}{72}+\frac{4}{24}-\frac{1}{12}\right)\,a+\frac{4}{180}-\frac{1}{72}-\frac{1}{120}\\
&=\frac{1}{120}\,a^2+\frac{1}{180}\,ab^2+\frac{1}{60}\,a.
\end{align*}
Finally, the coefficient of $ t ^ 6 $ is
\begin{align*}
&\tau_3+\frac{1}{720}-(\frac{\sigma_3}{2}-\frac{\sigma_2^2}{8})-\frac{1}{120}+\frac{\tau_2}{2}+\frac{\tau_1}{24}-\frac{\sigma_2}{12}\\
&=\left(\frac{16}{720}-\frac{4}{180}+\frac{1}{72} \right)\,a^2
+\left( \frac{4}{720}-\frac{1}{180}\right)\,ab^2\\&+\left( -\frac{8}{720}+\frac{17}{180}-\frac{4}{48}\right)\,a+\left(\frac{2}{720}-\frac{4}{180}+\frac{1}{72}-\frac{1}{120}-\frac{1}{36}+\frac{1}{24} \right)\\
&=\frac{1}{72}\,a^2.
\end{align*}Therefore
\[
	\|\ee^{-tM_{a,b}}\|^2 = 1 - \frac{a}{6}t^3 + \left(\frac{1}{120}\,a^2+\frac{1}{180}\,ab^2+\frac{1}{60} \,a\right)t^5 +\frac{a^2}{72}\,t^6+ \BigO(\langle A\rangle^4 t^7).
\]
and taking a square root, using the Taylor expansion $ \sqrt{1 + x} = 1 + x / 2 -x^2/8 + \BigO (x^3) $, we end the proof of Proposition \ref{prop_Taylor}.
\end{proof}
\subsection{Asymptotics as $ b $ tends to infinity.}
If $ M_ {a, b} $ were self-adjoint, we would have the equality $ \| \ee^{- tM_{a,b}} \|^2 = \ee^{- (1- c_1) t} $ for every $ t \geq 0 $. To study the extent to which this norm differs with the self-adjoint case, we compare $\|\ee^{-tM_{a,b}}\| $ with $\ee^{-(1-c_1)t/2}=\ee^{-t\Re\lambda_1}$, where $\lambda_1 = \frac{1}{2}(1-\ii b - c_1 - \ii c_2)$ is an eigenvalue of $ M_ {a, b} $ with minimal real part. In particular, we will study the regime when $ \vert b \vert \to + \infty $, corresponding to a large magnetic field.
\begin{proof}[Proof of Proposition \ref{prop_large_b}]
According to Theorem \ref{thm_2,4}, we have 
$$\|\ee^{-tM_{a,b}}\|^2 = \frac{1}{|A|}\ee^{-t}\left(T + \sqrt{S}\right),$$
where $ T $ and $ S $ are defined in \eqref{eq_def_S} and \eqref{eq_def_T}. Multiplying the previous equality by $ \ee^{2t \Re \lambda_1} $, we get
\begin{align*}
	\ee^{2t\Re \lambda_1}\|\ee^{-tM_{a,b}}\|^2 &= \ee^{2t\Re\lambda_1}\frac{\ee^{-t}}{|A|}(T + \sqrt{S})
	\\ &= \frac{\ee^{-c_1 t}}{|A|}\,T + \sqrt{\frac{\ee^{-2c_1 t}}{|A|^2}\,S}.
\end{align*}
We consider the regime when $ b \to + \infty $. We recall that
$$ T=\frac{1}{2}\left((|A| -A_1+2)\cosh(c_1 t) + (|A|+A_1-2)\cos(c_2 t)\right),$$
We expand out
\begin{equation}\label{eq_T_expand}
	\frac{\ee^{-c_1t}}{|A|}T = \frac{1}{4}\left(1 + \frac{2-A_1}{|A|}\right)(1+\ee^{-2c_1 t}) + \frac{1}{2}\left(1 - \frac{2-A_1}{|A|}\right)\ee^{-c_1 t}\cos c_2 t
\end{equation}
Note that
\[
	|A|^2 = (1-4a-b^2)^2 + 4b^2 = b^4\left(1+2(1+4a)b^{-2} + \BigO(\langle a\rangle^2 b^{-4})\right),
\]
so the Taylor expansion of the square root function gives (for $\langle a\rangle b^{-2}$ sufficiently small)
\[
	|A| = b^2\left(1+(1+4a)b^{-2} + \BigO(\langle a\rangle^2 b^{-4})\right).
\]
Along with the geometric series, we obtain
\begin{equation}\label{eq_2-A1AA}
	\frac{2-A_1}{|A|} = \frac{1 + (1+4a)b^{-2}}{1+(1+4a)b^{-2} + \BigO(\langle a\rangle^2 b^{-4})} = 1 + \BigO(\langle a\rangle^2 b^{-4}).
\end{equation}
As $$ \Re \cosh (ct) = \cosh (c_1t) \cos (c_2t) ,$$
we have then
\begin{align*}
	S &= 4a - 8a\cosh c_1t \cos c_2t + \frac{1}{4}\left(|A|^2 - (2-A_1)|A| + 8a\right)\cos 2c_2t 
	\\ & \qquad + \frac{1}{4}(|A|^2 + |A|(2-A_1)+ 8a)\cosh 2c_1t - \frac{1}{2}|A|^2
\end{align*}
We can write $ | A |^{- 2} \ee^{- 2c_1t} S $ in the following form:
\begin{align*}
	\frac{\ee^{-2c_1 t}}{|A|^2}S &= \frac{4a}{|A|^2}\left(\ee^{-2c_1 t} - (\ee^{-c_1 t} + \ee^{-3c_1 t})\cos c_2t\right) + 
	\\ & \qquad + \frac{1}{4}\left(1-\frac{2-A_1}{|A|} + \frac{8a}{|A|^2}\right)\ee^{-2c_1 t}\cos 2c_2 t
	\\ & \qquad + \frac{1}{8}\left(1 + \frac{2-A_1}{|A|} + \frac{8a}{|A|^2}\right)(1 + \ee^{-4c_1 t}) - \frac{1}{2}\ee^{-2c_1 t}.
\end{align*}
We replace the last $\frac{1}{2}\ee^{-2c_1 t}$ with
\[
	\frac{1}{2}\ee^{-2c_1 t} = \frac{1}{4}\left(1 + \frac{2-A_1}{|A|} + \frac{8a}{|A|^2} + 1-\frac{2-A_1}{|A|} - \frac{8a}{|A|^2}\right)\ee^{-2c_1 t}
\]
and we recombine terms to obtain
\begin{equation}\label{eq_S_expand}
\begin{aligned}
	\frac{\ee^{-2c_1 t}}{|A|^2}S &=\frac{4a}{|A|^2}\left(2\ee^{-2c_1 t} - (\ee^{-c_1 t} + \ee^{-3c_1 t})\cos c_2t\right) + 
	\\ & \qquad + \frac{1}{4}\left(1-\frac{2-A_1}{|A|} + \frac{8a}{|A|^2}\right)\ee^{-2c_1 t}(\cos 2c_2 t - 1)
	\\ & \qquad + \frac{1}{8}\left(1 + \frac{2-A_1}{|A|} + \frac{8a}{|A|^2}\right)(1 - \ee^{-2c_1 t})^2.
\end{aligned}
\end{equation}
To obtain a time-independent estimate, we use \eqref{eq_T_expand}, \eqref{eq_2-A1AA}, and \eqref{eq_S_expand} to obtain
\begin{align*}
	\ee^{(1-c_1)t}\|\ee^{-tM_{a,b}}\|^2 - 1 &= \frac{\ee^{-c_1 t}}{|A|} T-1 + \sqrt{\frac{\ee^{-2c_1 t}}{|A|^2}S} 
	\\ &= \frac{1}{2}(1+\BigO(\langle a\rangle^2 b^{-4}))(\ee^{-2c_1 t}-1) + \BigO(\langle a\rangle^2 b^{-4}) 
	\\ & \qquad + \frac{1}{2}(1+\BigO(\langle a\rangle^2 b^{-4}))\sqrt{(1-\ee^{-2c_1 t})^2 + \BigO(\langle a\rangle^{2}b^{-4})}, 
\end{align*}
we note that we replaced $ 1 $ by
$$1=\frac{1}{2} \left(1+\frac{2-A_1}{|A|}+1-\frac{2-A_1}{|A|}\right) $$ in the previous equality.

If we set $x = 1-\ee^{-2c_1t} \in [0, 1)$, we have shown that
\[
	\ee^{(1-c_1)t}\|\ee^{-tM_{a,b}}\|^2 - 1 = \frac{1}{2}(-x + \sqrt{x^2 + \BigO(\langle a \rangle^2 b^{-4})}) + \BigO(\langle a \rangle^2b^{-4}).
\]
But for $r \in \Bbb{R} \backslash \{0\}$, the absolute value of the function
\[
	  \varphi (x)=-x + \sqrt{x^2 + r} = \frac{r}{x + \sqrt{x^2 + r}}
\]
is decreasing when $ x \geq 0 $ if $ r \geq 0 $ or $ x \geq (-r)^{1/2} $ if $ r <0 $. The maximum of $ \vert \varphi (x) \vert $ is obtained when $ r \geq 0 $ if $ x = 0 $ and if $ r <0 $, the maximum is reached when $ x = \vert r \vert^{1/2} $. In either case the maximum is $ \vert r \vert^{1/2} $, and we have shown
\[
	\left|\ee^{(1-c_1)t}\|\ee^{-tM_{a,b}}\|^2 - 1\right| = \BigO(\langle a \rangle b^{-2}).
\]
The lower bound $0 \leq \ee^{(1-c_1)t}\|\ee^{-tM_{a,b}}\|^2 - 1$ is trivial from testing on an eigenvector of $M_{a,b}$ with eigenvalue $\lambda_1 = \frac{1}{2}(1-\ii b - c_1 - \ii c_2)$.
This completes the proof of the Proposition \ref{prop_large_b}.
\end{proof}
\subsection{Asymptotics in long time.}
In this part, we calculate the exact value of $\lim\limits_{t\to +\infty}\, \ee^{(1-c_1)t}\,\Vert \ee^{-tM_{a,b}}\Vert^2$, which measures in a sense the norm of the spectral projector of $ M_{a, b} $ associated with the eigenvalue $ \lambda_1 $, where $\lambda_1 = \frac{1}{2}(1-\ii b - c_1 - \ii c_2)$ is one of the eigenvalues of $M$ with minimal real part.
\begin{proof}[Proof of Proposition \ref{prop_long_t}]
Using equality \eqref{eq_S_expand}, we get
\begin{align*}
	\ee^{-2c_1t}S &= \frac{1}{8}(|A|^2 + 8a + |A|(2-A_1))(\ee^{-4c_1t} + 1) - 4a(\ee^{-3c_1t} + \ee^{-c_1t})\cos c_2t 
	\\ & \qquad+ (4a - \frac{1}{2}|A|^2 + \frac{1}{4}(|A|^2 + 8a - |A|(2-A_1))\cos 2c_2 t)\ee^{-2c_1t}.
\end{align*}
We seek to obtain a precise estimate when $ t \to + \infty $, using
\[
	\frac{\ee^{-2c_1t}}{|A|^2}S = \frac{1}{8}\left(1 + \frac{2-A_1}{|A|^2} + \frac{8a}{|A|^2}\right) - \frac{1}{2} \ee^{-2c_1 t} + \BigO(\langle a\rangle^2 b^{-4}\ee^{-c_1t}).
\] 
Similarly, we can show
\[
	\frac{\ee^{-c_1 t}}{|A|}T = \frac{1}{4}\left(1 + \frac{2-A_1}{|A|}\right) + \frac{1}{2}\ee^{-2c_1t} + \BigO(\langle a\rangle^2 b^{-4}\ee^{-c_1 t}).
\]
By defining
\[
	R_0 = \frac{1}{8}\left(1 + \frac{2-A_1}{|A|} + \frac{8a}{|A|^2}\right)
\]
and by noting that $ R_0 = \frac {1}{4} (1+ \BigO (\langle a \rangle ^ 2 b ^ {- 4}) $, we get
\begin{align*}
	\ee^{2t\Re \lambda_1 }\|\ee^{-tM_{a,b}}\|^2 &= \frac{\ee^{-c_1 t}}{|A|} T + \sqrt{\frac{\ee^{-2c_1 t}}{|A|^2}S} 
	\\ &= \frac{1}{4}\left(1 + \frac{2-A_1}{|A|}\right) + \frac{1}{2}\ee^{-2c_1 t} + \BigO(\langle a\rangle^2 b^{-4}\ee^{-c_1 t})
	\\ & \qquad + \sqrt{R_0}\sqrt{1 - \frac{1}{2R_0}\ee^{-2c_1 t} + \BigO(\langle a\rangle^2 b^{-4}\ee^{-c_1 t})}\\
	&=R_1+\left( \frac{1}{2}-\frac{1}{4\sqrt{R_0}}\right)\ee^{-2c_1t}+\BigO(\langle a\rangle^2 b^{-4}\ee^{-c_1 t})	,
\end{align*}
where $ R_1 $ is defined by
$$R_1=\frac{1}{4}\left(1 + \frac{2-A_1}{|A|}\right) +\sqrt{R_0}. $$
By the Taylor expansion of the square root function, we obtain that there is $ C> 0 $ such that if
$$E(t)=\ee^{-2c_1t}+\langle a\rangle^2\,b^{-4}\ee^{-c_1t}\leq \displaystyle \frac{1}{C}, $$
 then
\begin{align}
	\left|\frac{1}{\sqrt{R_1}}\ee^{t\Re \lambda_1}\|\ee^{-tM_{a,b}}\|-1\right| \leq CE(t).\label{equality_tlong.3}
\end{align}
Now, we try to simplify the expressions of $ R_0 $ and $ R_1 $. For this we need the following lemma:
\begin{lem}\label{lem_simp_1.3}
Let $ a> 0 $ and $ b \neq 0 $.
\begin{align}
\vert A\vert +2-A_1&=2(c_2^2+1),\label{simp_1.3}\\
4a&=(1+c_2^2)(1-c_1^2),\label{simp_2.3}\\
\vert \lambda_1\vert^2&=(1+c_2^2) (1-c_1)^2 /4\label{simp_3.3}.
\end{align}
\end{lem}
\begin{proof}
We start by showing  equality \eqref{simp_1.3}. Using the system of equations \eqref{sys 1},
\begin{align*}
\vert A\vert +2-A1=c_1^2+c_2^2+2-c_1^2+c_2^2=2(1+c_2^2).
\end{align*}
Then, we pass to show equality \eqref{simp_2.3}. Using the system of equations \eqref {sys 1} and the definition of $ A_1 = -b ^ 2-4a + 1 $, we obtain
\begin{align*}
4a=-b^2-A_1+1&=-c_1^2c_2^2-c_1^2+c_2^2+1\\
&=c_2^2(1-c_1^2)+1-c_1^2\\
&=(1-c_1^2)(1+c_2^2).
\end{align*}
Finally, we will show equality \eqref{simp_3.3}. Using the definition of $ \lambda_1 $ as a function of $ c_1, c_2 $ and $ b $ and replacing $ b $ by $ -c_1c_2 $, we have
\begin{align*}
\vert \lambda_1\vert^2&=\frac{1}{4}\left( (1-c_1)^2+(b+c_2)^2 \right)\\
&=\frac{1}{4}\left( (1-c_1)^2+(-c_1c_2+c_2)^2 \right)\\
&=\frac{1}{4}\left( (1-c_1)^2+c_2^2(1-c_1)^2 \right)\\
&=\frac{1}{4}(1+c_2^2)(1-c_1^2).
\end{align*}
This completes the demonstration of Lemma \ref{lem_simp_1.3}.
\end{proof}
We return to the proof of Proposition \ref{prop_long_t}. Using the equations in Lemma \ref{lem_simp_1.3}, $ R_0 $ takes the following form:
\begin{align*}
R_0 &= \frac{1}{8}\left(1 + \frac{2-A_1}{|A|} + \frac{8a}{|A|^2}\right)\\
&= \frac{1}{8\vert A\vert }\left(\vert A\vert +2-A_1 + \frac{8a}{|A|}\right)\\
&=\frac{(1+c_2^2)^2}{4\vert A\vert^2}.
\end{align*}
So $ \sqrt{R_0} = \displaystyle \frac{1 + c_2 ^ 2}{2 \vert A \vert} $. We end this proof by simplifying the expression of $ R_1 $ as follows:
\begin{align*}
R_1&=\frac{1}{4}\left(1 + \frac{2-A_1}{|A|}\right) +\sqrt{R_0}\\
&=\frac{1}{4\vert A\vert }(\vert A\vert+2-A_1)+\frac{c_2^2+1}{2\vert A\vert}\\
&=\frac{c_2^2+1}{\vert A\vert}\\
&=\frac{c_2^2+1}{c_1^2+c_2^2}.
\end{align*}
Inserting the simplified form of $ R_1 $ in \eqref{equality_tlong.3} completes the demonstration of Proposition \ref{prop_long_t}.
\end{proof}
\appendix
\section{Periodicity of $ \ee ^ {t \Re \lambda_1} \ee^{- t M_{a, 0}} $ when $ b = 0 $ and $ a> 1/4 $}
   When the magnetic field is zero and when $a > 1/4$, we notice that $ \ee^{t \Re \lambda_1} \ee^{- t M_{a, 0}} $ is periodic, as can be guessed from Figure \ref{b0_2} (which displays $\log \|\ee^{- t M_{a, 0}}\|$ for $a = 24 > 1/4$). Considering a similar graph for $ a = 0.2 < 1/4 $, the phenomenon of periodicity has disappeared. The aim of this section is to confirm this numerical observation and to express the period as a function of electrical parameter $a$, explaining where the hypothesis $ a> 1/4 $ comes into play.
   \begin{figure}[!h]
\begin{minipage}[c]{1\linewidth}
   \centering
      \includegraphics[width=.9\linewidth]{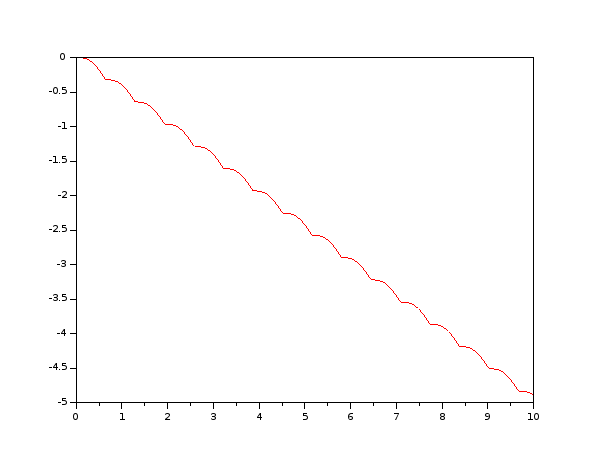} 
   \caption{Time behavior of $\log (\Vert \ee^{-t\,M_{a,0}} \Vert)$ with $ b = 0 $ and $ a = 24> 1/4 $}
   \label{b0_2}
   \end{minipage} \hfill
   \begin{minipage}[c]{1\linewidth}
     \includegraphics[width=.9\linewidth]{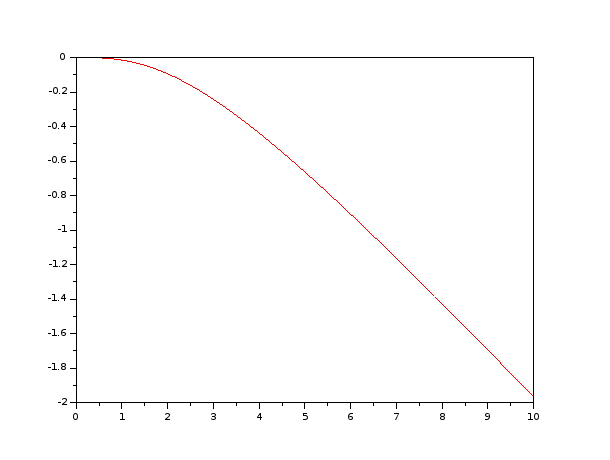} 
      \caption{Time behavior of $\log (\Vert \ee^{-t\,M_{a,0}} \Vert)$ with $b=0$ and $a=0.2<1/4$}
   \end{minipage} 
    \end{figure} 
    \begin{prop}
    Let $ a> 1/4 $. Then $ \ee^{t \Re \lambda_1} \ee^{- t M_{a, 0}} $ is periodic with period $T=\displaystyle \frac{4\pi}{\sqrt{4\,a-1}}$.
    \end{prop}
    \begin{proof}
    We recall the matrix $ M_{a,0} $ associated with the Kramers-Fokker-Planck operator when $ b = 0 $:
    $$ M_{a,0}=\begin{pmatrix}
1&0&\sqrt{a}&0\\
0&1&0&\sqrt{a}\\
-\sqrt{a}&0&0&0\\
0&-\sqrt{a}&0&0
\end{pmatrix} \,.$$
The eigenvalues of $ M_{a, 0} $ are $ \lambda_1 $ and $ \lambda_2 $, each with multiplicity $ 2 $, explicitly given by
$$\lambda_1=\frac{1}{2}(-\sqrt{1-4\,a}+1) \text{ and } \lambda_2=\frac{1}{2}(\sqrt{1-4\,a}+1)\,. $$
The matrix $ M_ {a, 0} $ is diagonalizable, and we let $ P $ denote the change-of-basis matrix from the canonical basis to the basis of eigenvectors,
$$ P=\begin{pmatrix}
0&-\frac{\lambda_1}{\sqrt{a}}&0&-\frac{\lambda_2}{\sqrt{a}}\\
-\frac{\lambda_1}{\sqrt{a}}&0&-\frac{\lambda_2}{\sqrt{a}}&0\\
0&1&0&1\\
1&0&1&0
\end{pmatrix}\,.$$
The matrix $ M_{a, 0} $ can be written as
\begin{align}
M_{a,0}=P\,D\,P^{-1} \text{ with } D= \mathrm{Diag}(\lambda_1,\lambda_1,\lambda_2,\lambda_2)\,.
\end{align}
To show that $ \ee^{t \Re \lambda_1} \ee^{- t  M_{a, 0}} $ is periodic, we must show that there exists $ t_{0}> 0 $ such that
$$\ee^{t_0/2}\ee^{-t_0\,M_{a,0}}=I_4\,, $$
because $\Re \lambda_1 =1/2$, according to Proposition \ref{prop exp}, we have that
$$ \ee^{-t\,M_{a,0}}= P\,\ee^{-t\,D}\,P^{-1}\,,$$
therefore, the question amounts to showing the periodicity of the matrix $ \ee^{ t / 2} \ee^{- t D} $ in the sense introduced above. That is to say we are looking for a real $ t_0> 0 $ such that
\begin{align*}
\ee^{-t_0\,D}=\ee^{-t_0/2}I_4\,,
\end{align*}
and using the fact that $$\ee^{-t_0\,D}=\mathrm{Diag} (\ee^{-t_0\lambda_1}, \ee^{-t_0\lambda_2},\ee^{-t_0\lambda_1},\ee^{-t_0\lambda_2})\,,$$
we observe that
\begin{align*}
\ee^{-t_0\,D}=\ee^{-t_0/2}I_4
&\Longleftrightarrow \ee^{\frac{t_0}{2}(i\,\sqrt{4\,a-1} -1)}=\ee^{-\frac{t_0}{2}}\,,\\
&\Longleftrightarrow \frac{t_0}{2}\,\sqrt{4\,a -1}=2\,k\,\pi\text{ with } k\in \mathbb{Z}\,,\\
&\Longleftrightarrow t_0=\frac{4\,k\,\pi}{\sqrt{4\,a-1}}\,.
\end{align*}
In particular, we take $ k = 1 $, we get $T=\frac{4\,\pi}{\sqrt{4\,a-1}}$, this shows the periodicity of $ \ee^{- t_0 / 2} \ee^{- t M_{a, 0}} $ when $ a> 1/4 $ with period $ T $.
    \end{proof}
    \section{Numerical illustrations of main results}
\subsection{Spectral abscissa of $ M_ {a, b} $.}
In the table below we calculate numerically, using Scilab, the values of the spectral abscissa of the matrix $ M_{a, b} $ which exactly equals $ \Re \lambda_1 = (1-c_1) / 2 $ in comparing those with the values of $ a / b^2 $ by setting $ a = 14 $ and taking several values of $ b $.
\begin{center}
\begin{tabular}{|l|l|l|}
\hline  b & Spectral abscissa  & $a/b^2$\\
\hline  5& 0.221337 & 0.56  \\
\hline 10 & 0.0992201 &0.14\\
\hline 100 & 0.0013940 & 0.0014\\
\hline 200 & 0.0003496 & 0.00035\\
\hline 800 & 0.0000219 & 0.0000219\\
\hline
\end{tabular}
\end{center}
We observe in the previous table that when $ b $ increases, the spectral gap of the operator $ P_{a, b} $, which coincides with the spectral abscissa of the matrix $ M_ {a, b} $, approaches $ a / b ^ 2 $. In particular this confirms the asymptotics when $ b \to + \infty $ of $  \Re \lambda_1 $. Physically, this represents the rate of return to the equilibrium, and our result therefore serves to quantify the influence of a large magnetic field slowing down the rate of return to the equilibrium.

 \subsection{Regime when $ t \to 0^+ $.}
 In this part, we will give some numerical illustrations in small time of the norm of $ \Vert \ee^{- tM_{a, b}} \Vert $. This is the rate of return to the equilibrium for $ \ee^{- tP_{a, b}} $ in a regime where $ \Vert \ee^{- tM_{a, b}}  \Vert $ is significantly greater than $ \ee^{- t \Re \lambda_1} $. By observing Figures \ref{f1} and \ref {f2}, where we draw the exact exponential norm and its associated approximations given in Proposition \ref{prop_Taylor} with a polynomial error of order $ t^7 $.
We note that when we increase the magnetic field, the error increases, which is to be expected because our approximations are taken in the regime where $ \langle A \rangle t^2 $ is sufficiently small. Then, in Figures \ref{f3} and \ref{f4}, we draw the exact norm with the polynomial given in Proposition \ref{prop_Taylor} of order $ 6 $. We observe that when we increase the magnetic field the error increases (see Figure \ref{f4}).
 \begin{figure}[!h]      
   \begin{minipage}[c]{1\linewidth}
   \centering
  \includegraphics[width=.95\linewidth]{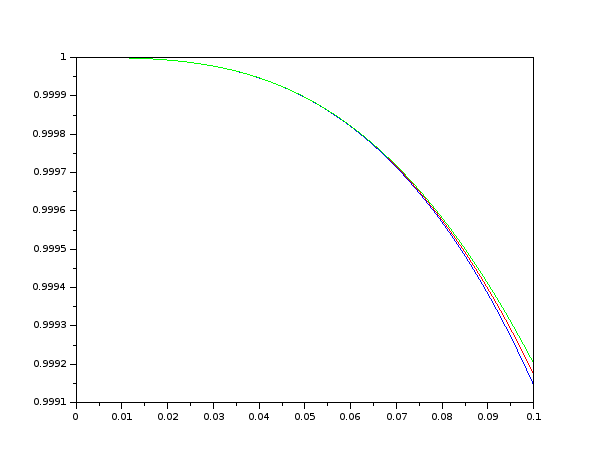} 
 \caption{Comparison between the behavior in time of the norm $ \Vert \ee^{- tM_{a, b}} \Vert $ and their approximations in small time of the errors ($ - \langle A \rangle^3 \, t^7 $ and $ \langle A \rangle^3 \, t^7) $ with $ b = 2 $ from left to right.}
   \label{f1}
   \end{minipage}      
   \begin{minipage}[c]{1\linewidth}
     \includegraphics[width=.95\linewidth]{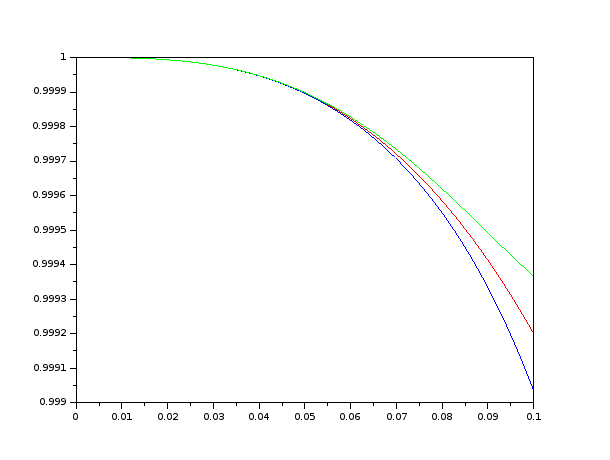}
       \caption{Comparison between the behavior in time of the norm $ \Vert \ee^{- tM_{a, b}} \Vert $ and their approximations in small time of the errors ($ - \langle A \rangle^3 \, t^7 $ and $ \langle A \rangle^3 \, t^7) $ with $ b = 10 $ from left to right.}
      \label{f2} 
      \end{minipage}      

    \end{figure} 
    
\begin{figure}[!h]
\begin{minipage}[c]{1\linewidth}
  \centering
    \includegraphics[width=1\linewidth]{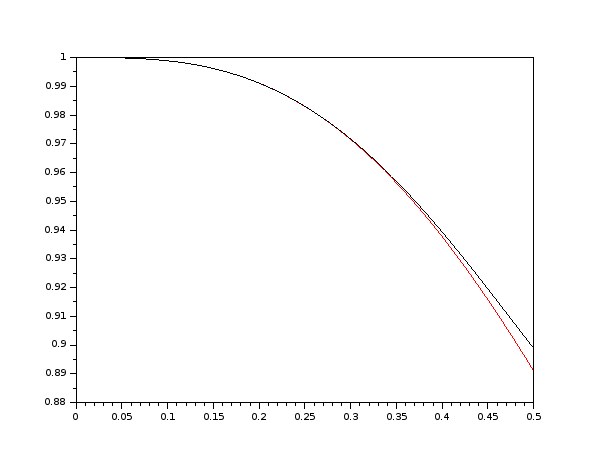} 
 \caption{Comparison between the behavior in time of the norm $ \Vert \ee^{- t  M_{a, b}} \Vert $ and its limited development when $ t \to 0 $ to the order $ 6 $ when $ b = 2$.}
   \label{f3}
   \end{minipage}      
   \begin{minipage}[c]{1\linewidth}
   \centering
   \includegraphics[width=1\linewidth]{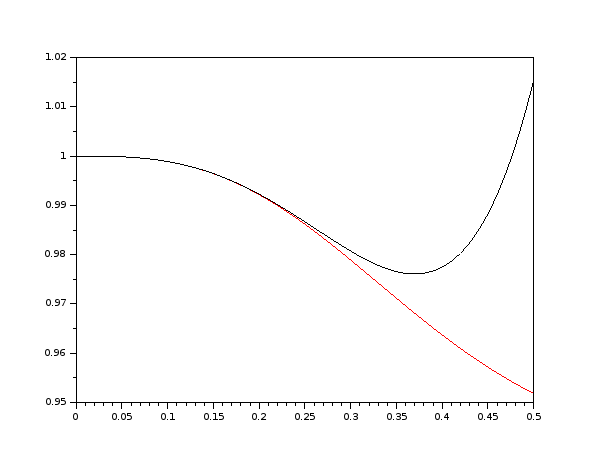} 
 \caption{Comparison between the behavior in time of the norm $ \Vert \ee^{- t  M_{a, b}} \Vert $ and its limited development when $ t \to 0 $ to the order $ 6 $ when $b=10$.}
   \label{f4}
   \end{minipage}
    \end{figure} 
   \subsection{Regime when $ b \to + \infty $.}
   In this part, we will give numerical illustrations which illustrate the uniform estimates obtained in Proposition \ref{prop_large_b}. Above all, we compare the deviations
   $$ \Vert \ee^{-tP_{a,b}}-\Pi_0\Vert=\Vert \ee^{-tM_{a,b}}\Vert,$$
with the approximations given in Proposition \ref{prop_large_b}. We observe numerically in Figures \ref{f5} and \ref{f6} that when we increase the magnetic field the approximation becomes more precise. In examining the figures, it seems that when we increase the magnetic field the norm of the matrix exponential more closely resembles the self-adjoint prediction which equals $ \ee^{- t \Re \lambda_1} $. In addition, we compare on a logarithmic scale in Figure \ref{f9} the exponential norm when $ b = 0 $ and $ b \neq 0 $. We observe that the non-self-adjoint character seems to disappear when $ b $ increases.   
   
   To study the return to the equilibrium, we represent in Figure \ref{f10} the norm $ \Vert \ee^{- tP_{a, b}} - \Pi_0 \Vert $ by taking several values of the magnetic parameter $ b $ between $ 20 $ and $ 100 $. We notice that when $ b = 20 $ the return to equilibrium appears at a time between $ 200 $ and $ 300 $, while when $ b = 100 $ on a time scale is equal to $ 500 $, showing that the return to the equilibrium is weaker. Consequently, we see that the magnetic field slows the return to equilibrium.
    \begin{figure}[!h]    
   \begin{minipage}[c]{1\linewidth}
   \centering{
   \includegraphics[width=.9\linewidth]{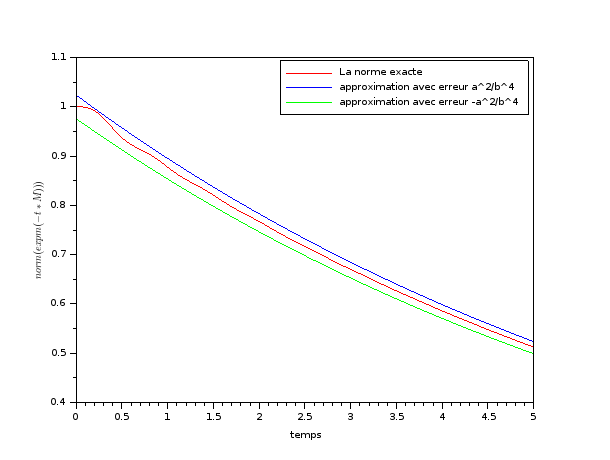} 
	}
 \caption{Comparison between the behavior in time of the norm $ \Vert \ee^{- t M_{a, b}} \Vert $ and its approximations associated with Proposition \ref{prop_large_b} when $ b = 8 $.}
   \label{f5}
 \end{minipage}      
   \begin{minipage}[c]{1\linewidth}
      \centering{ 
      \includegraphics[width=.9\linewidth]{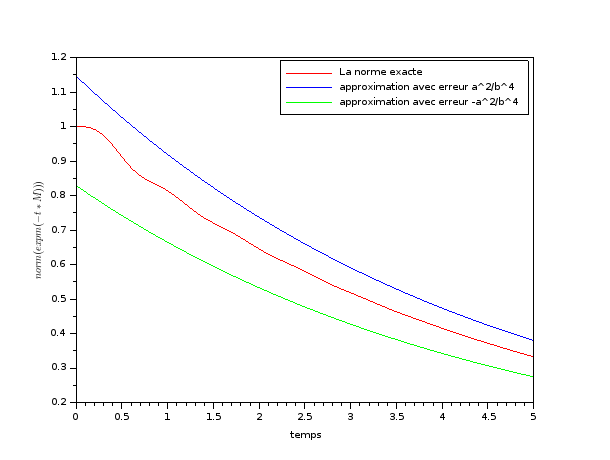}
      }
       \caption{Comparison between the behavior in time of the norm $ \Vert \ee^{- t M_{a, b}} \Vert $ and its approximations associated with Proposition \ref{prop_large_b} when $ b = 5 $.}
      \label{f6}
\end{minipage}      
 
    \end{figure} 
    
     \begin{figure}[!h]
   \begin{minipage}[c]{1\linewidth}
   \centering{
   \includegraphics[width=.9\linewidth]{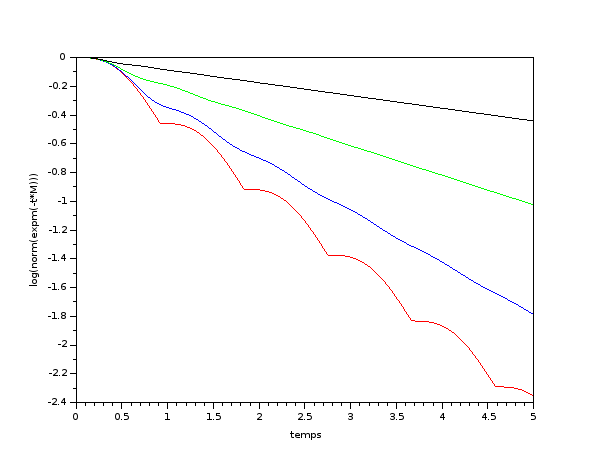} 
	}
 \caption{Comparison between the time behavior of $\log (\Vert \ee^{- t M_{a, b}} \Vert) $ when $ b = 0$ and when $b \neq 0$ from left to right}
   \label{f9}
   \end{minipage}      
   \begin{minipage}[c]{.9\linewidth}
  \centering{ 
 \includegraphics[width=1\linewidth]{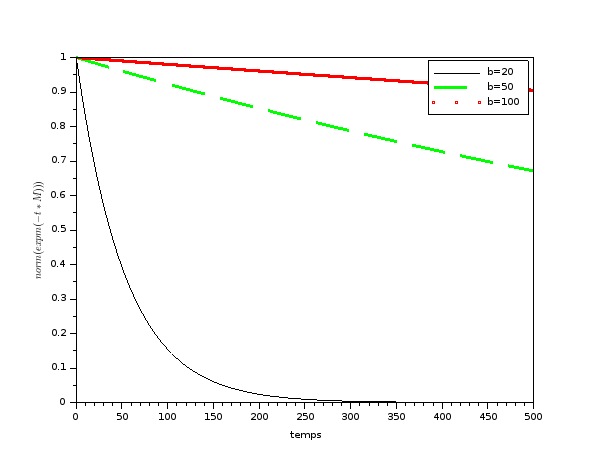}
  }
  \caption{Time behavior of $ \Vert \ee^{- t M_{a, b}} \Vert $ in several values of $ b =20,  50$ and $ 100 $.}
  \label{f10}
  \end{minipage}      
  \end{figure}
  \subsection{Long time regime}
  In this part we will give numerical illustrations which illustrate the result obtained in Proposition \ref{prop_long_t}.
  As mentioned in the introduction, the result of Proposition \ref{prop_long_t} shows that there is $ C> 0 $ such that if we assume that
  $$ E(t) = \ee^{-2c_1\,t} + \langle a\rangle^2 b^{-4}\ee^{-c_1t} \leq \frac{1}{C},$$ then $$ \left|\left( \frac{c_1^2+c_2^2}{c_2^2+1}\right)^{1/2}\ee^{\frac{t(1-c_1)}{2}}\|\ee^{-tM_{a,b}}\|-1\right| \leq CE(t).$$  
  One can show that the norm of the spectral projector $ \Pi_1 $  associated with $ \lambda_1 $ is
  $$\Vert \Pi_1\Vert=\lim\limits_{t\to +\infty}\ee^{\frac{t(1-c_1)}{2}}\|\ee^{-tM_{a,b}}\|= \sqrt{\frac{1+c_2^2}{c_1^2+c_2^2}}=\sqrt{R_1}. $$
  In Figure \ref{f7}, we draw the behavior in time of the function $ \ee^{\frac{t (1-c_1)}{2}} \| \ee^{- tM_{a, b} } \| $ and the norm exactof the spectral projector $ \Pi_1 $ when the electric parameter $ a $ is $ 8 $ and the magnetic parameter $ b $ is $ 12 $. We can observe that over time the function $ \ee^{\frac{t (1-c_1)}{2}} \| \ee^{- tM_{a, b}} \| $ gets closer to the exact norm of $ \Pi_1 $.
  Then, in Figures \ref{f8} and \ref{f15}, we compare the function $ \ee^{\frac{t (1-c_1)}{2}} \| \ee^{- tM_{a, b}} \| $ and the approximations $ \sqrt{R_1} (1-CE (t)) $ and $ \sqrt{R_1} (1 + CE (t)) $ when $ a = 8 $ and $ b =12 $  or $ 5$ respectively with the constant $ C =50 $. Note that in the two figures the curve of the function $ t \to \ee^{\frac{t (1-c_1)}{2}} \| \ee^{- tM_{a, b}} \| $ lies well between the two curves associated with the approximations cited just before. In addition, from a certain time $ t $ between $ 2.5 $ and $ 3 $ for Figure \ref{f8} (when $ b = 12 $), we see that the error between the three curves becomes very small. Whereas in Figure \ref{f15} and when $ b = 5 $, a similar decrease of the error appears when the time $ t $ is between $ 4 $ and $ 5 $ instead. In conclusion, the estimate given in Proposition \ref{prop_long_t} becomes more precise when the magnetic field $ b $ is increased.
  
  \section*{Acknowledgments} The author is grateful to Joe Viola for his important continued help and advice throughout the creation of this work. The author thanks also the Centre Henri Lebesgue ANR-11-LABX-0020-01 for his support.
   \begin{figure}[!h]    
   \begin{minipage}[c]{1\linewidth}
    \includegraphics[width=.9\linewidth]{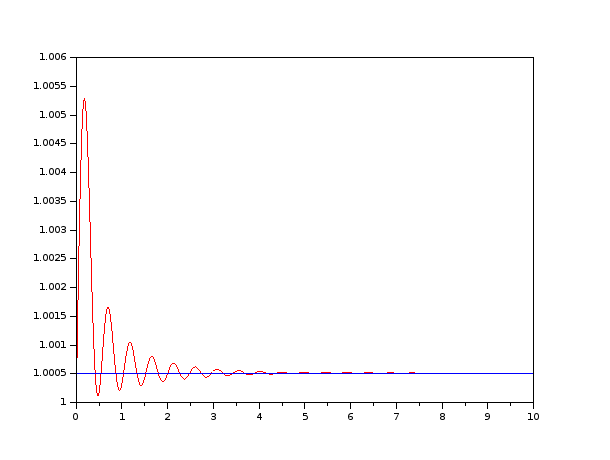} 
 \caption{Comparison between the behavior in time of the norm $\ee^{t\Re \lambda_1}\Vert \ee^{-t\,M_{a,b}}\Vert$ and $\displaystyle\sqrt{\frac{c_2^2+1}{c_1^2+c_2^2}}$ when $a=8$ and $b=12$.}
   \label{f7}
  \end{minipage}      
   \begin{minipage}[c]{1\linewidth}
      \includegraphics[width=.9\linewidth]{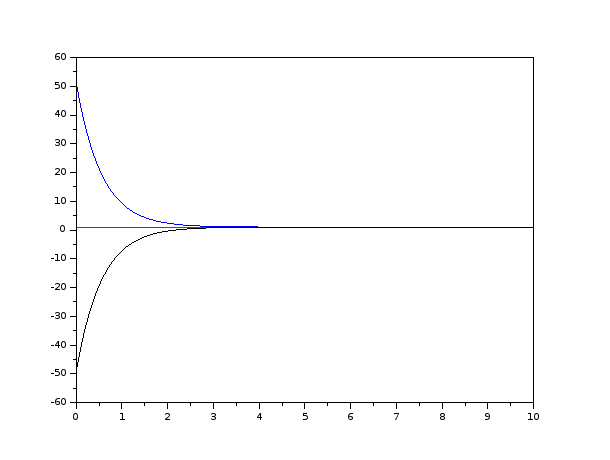}
       \caption{Comparison between the behavior in time of the norm $ \ee^{t \Re \lambda_1} \Vert \ee^{- t  M_{a, b}} \Vert $ and its approximations given in Proposition \ref{prop_long_t} with $ a = 8 $ and $ b = 12 $ from left to right when $ C = 50 $.}
      \label{f8}
      \end{minipage}     
   \end{figure}
\begin{figure}[!h]
      \begin{minipage}[c]{1\linewidth}
     \includegraphics[width=1\linewidth]{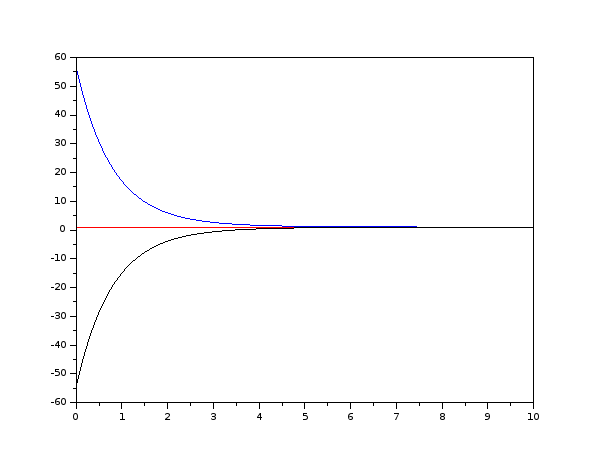}
       \caption{Comparison between the behavior in time of the norm $ \ee^{t \Re \lambda_1} \Vert \ee^{- t  M_{a, b}} \Vert $ and its approximations given in Proposition \ref{prop_long_t} with $ a = 8 $ and $ b = 5$ from left to right when $ C = 50 $.}
      \label{f15}
      \end{minipage}     
   \end{figure}
\bibliographystyle{alpha}
\bibliography{library}
\addcontentsline{toc}{section}{Bibliographie} 
\end{document}